\documentclass[a4paper]{article} \usepackage{etoolbox} \newtoggle{author} \toggletrue{author}
\newtoggle{lineno} \togglefalse{lineno}

\iftoggle{author}{
    \usepackage[utf8]{inputenc}
    \usepackage[T1]{fontenc}
    \usepackage[english]{babel}

    \usepackage{amsmath}
    \usepackage{amssymb}
    \usepackage{amsthm}
    \usepackage{cite}
    \usepackage{enumerate}
    \usepackage{geometry}
    \usepackage{hyperref}
    \usepackage{microtype}
}{
    \usepackage[super]{cite}
    \usepackage{url}
}

\iftoggle{lineno}{\usepackage{lineno}}{}

\usepackage[dvipsnames]{xcolor}

\usepackage{bm}
\usepackage{tikz}
\usepackage{todonotes}
\usepackage{upgreek}

\usepackage[capitalize]{cleveref}

\newcommand{\MM}{\mathbb{M}}
\newcommand{\NN}{\mathbb{N}}
\newcommand{\PP}{\mathbb{P}}
\newcommand{\RR}{\mathbb{R}}

\newcommand{\bA}{\mathbf{A}}

\newcommand{\bn}{\mathbf{n}}

\newcommand{\bgamma}{\bm \upgamma}

\newcommand{\cE}{\mathcal{E}}
\newcommand{\cF}{\mathcal{F}}

\newcommand{\cT}{\mathcal{T}}
\newcommand{\cV}{\mathcal{V}}

\newcommand{\rv}{{\rm v}}

\newcommand{\<}{\langle}
\renewcommand{\>}{\rangle}

\DeclareMathSymbol{\epsilon}{\mathord}{letters}{"22}
\DeclareMathSymbol{\phi}{\mathord}{letters}{"27}

\DeclareMathOperator{\Lagrange}{L}
\DeclareMathOperator{\SZ}{SZ}
\DeclareMathOperator*{\argmin}{argmin}

\DeclareMathOperator{\ddiv}{div}
\DeclareMathOperator{\intt}{int}
\DeclareMathOperator{\sgn}{sgn}
\DeclareMathOperator{\spann}{span}
\DeclareMathOperator{\tr}{tr}

\iftoggle{author}{
    \newcommand{\refcite}[1]{\cite{#1}}

    \newenvironment{itemlist}{\begin{itemize}}{\end{itemize}}
    \newenvironment{romanlist}[1][]{\begin{enumerate}[#1]}{\end{enumerate}}

    \theoremstyle{definition}
    \newtheorem{definition}{Definition}[section]
    \newtheorem{assumptions}{Assumptions}[section]

    \theoremstyle{plain}
    \newtheorem{corollary}{Corollary}[section]
    \newtheorem{lemma}{Lemma}[section]
    \newtheorem{proposition}{Proposition}[section]
    \newtheorem{theorem}{Theorem}[section]

    \theoremstyle{remark}
    \newtheorem{remark}{Remark}[section]

    \numberwithin{equation}{section}
}{
    \newtheorem{assumptions}{Assumptions}[section]
}

\crefname{equation}{}{}
\iftoggle{author}{\crefname{figure}{Figure}{Figures}}{}
\crefname{section}{section}{sections}

\iftoggle{lineno}{
    \cspreto{equation*}{\linenomath}
    \cspreto{endequation*}{\endlinenomath}
    \cspreto{align}{\linenomath}
    \cspreto{endalign}{\endlinenomath}
    \cspreto{align*}{\linenomath}
    \cspreto{endalign*}{\endlinenomath}

    \linenumbers
}{}

\begin{document}
\iftoggle{author}{}{
    \markboth{Guillaume Bonnet, Andrea Cangiani and Ricardo Nochetto}{Conforming virtual element method for nondivergence form linear elliptic equations}

    \catchline{}{}{}{}{}
}

\title{Conforming virtual element method for nondivergence form linear elliptic equations with Cordes coefficients}

\iftoggle{author}{
    \author{
        Guillaume Bonnet\footnote{CEREMADE, CNRS, Université Paris-Dauphine, Université PSL, 75016 Paris, France}
        \and
        Andrea Cangiani\footnote{Mathematics Area, International School for Advanced Studies (SISSA), 34136 Trieste, Italy}
        \and
        Ricardo H. Nochetto\footnote{Department of Mathematics, University of Maryland, College Park, MD 20742}
    }
}{
    \author{Guillaume Bonnet}

    \address{
        CEREMADE, CNRS, Université Paris-Dauphine, Université PSL \\
        place du Maréchal de Lattre de Tassigny, 75016 Paris, France \\
        bonnet@ceremade.dauphine.fr
    }

    \author{Andrea Cangiani}

    \address{
        Mathematics Area, International School for Advanced Studies (SISSA) \\
        via Bonomea, 265, 34136 Trieste, Italy \\
        acangian@sissa.it
    }

    \author{Ricardo H. Nochetto}

    \address{
        Department of Mathematics, University of Maryland, College Park \\
        4176 Campus Drive~-- William E.~Kirwan Hall, College Park, MD 20742 \\
        rhn@umd.edu
    }
}

\maketitle

\iftoggle{author}{}{
    \begin{history}
        \received{(Day Month Year)}
        \revised{(Day Month Year)}
        %\accepted{(Day Month Year)}
        \comby{(xxxxxxxxxx)}
    \end{history}
}

\begin{abstract}
    We propose and analyze an $H^2$-conforming Virtual Element Method (VEM) for the simplest linear elliptic PDEs in nondivergence form with Cordes coefficients. The VEM hinges on a hierarchical construction valid for any dimension $d\ge 2$. The analysis relies on the continuous Miranda-Talenti estimate for convex domains $\Omega$ and is rather elementary. We prove stability and error estimates in $H^2(\Omega)$, including the effect of quadrature, under minimal regularity of the data. Numerical experiments illustrate the interplay of coefficient regularity and convergence rates in $H^2(\Omega)$.
\end{abstract}

\iftoggle{author}{}{

}

%%%%%%%%%%%%%%%%%%%%%%%%%%%%%%%%%%%%%%%%%%%%%%%%%%%%%%%%%%%%%%%%%%%%%%%%%%%%%%%%%%%%%%%%

\section{Introduction}

%\todo[inline]{TODO: update the introduction. Include comparisons with other type of solutions and discretisations like discontinuous Galerkin and with finite difference / viscosity solution approaches. Discuss Neumann / oblique derivative boundary conditions?}

%In this paper, we consider the problem of the numerical approximation of strong solutions to the elliptic equations in the nondivergence form
In this paper, we study the discretization by an $H^2$-conforming Virtual Element Method (VEM) of the simplest elliptic equations in nondivergence form
\begin{equation}
    \label{eq:main}
    \begin{cases}
        A : \nabla^2 u = f &\text{in } \Omega, \\
        u = g &\text{on } \partial \Omega,
    \end{cases}
\end{equation}
where $\Omega \subset \RR^d$ is a bounded convex polytopal domain, $A \in [L^\infty(\Omega)]^{d \times d}$ is a field of uniformly elliptic symmetric positive definite matrices, $f \in L^2(\Omega)$, and $g \in H^2(\Omega)$. Here, $\nabla^2$ denotes  the Hessian and the colon stands for the Frobenius inner product between matrices.

The value function of a stochastic differential equation satisfies a linear PDE such as \cref{eq:main}, typically with low order terms\cite{cerrai2001,cherny2005}. They also appear in the linearization of fully nonlinear equations, such as Monge-Amp\`ere and Hamilton-Jacobi-Bellman, which are relevant to a number of applications, including differential geometry, optimal transport and stochastic control, fluid mechanics and meteorology, image processing\cite{feng2013,dephilippis2014,falcone2013}.

The structure of \cref{eq:main} is deceivingly simple. For example, \cref{eq:main} with forcing
$f=0$ and discontinuous coefficient $A$ given by
\[
    A(x) = I_{d \times d} + \frac{d + \alpha -2}{1-\alpha}
    \frac{x}{|x|} \otimes \frac{x}{|x|}
\]
admits two solutions in the unit ball $B_1(0)$ centered at $0$, namely
$u(x) = |x|^{\alpha} - 1$ and $u(x)= 0$, which
happen to be of class $H^2(\Omega)$ provided $d>2(2-\alpha)$ for any
$0<\alpha<1$. For $0\le \mu < 1$, the (pointwise) Cordes condition
\[
\frac{|\bA|^2}{(\tr \bA)^2} \le \frac{1}{d-\mu^2}
\]
gives sufficient conditions on $A$, possibly discontinuous, to guarantee a unique solution $u\in H^2(\Omega)$ of \cref{eq:main} for $\Omega$ convex and $f\in L^2(\Omega)$; this is a consequence of a perturbation theory for the Laplacian\cite{campanato1993,maugeri2000}.
We refer to Ref.~\refcite{nochetto2018} for a brief review on the different notions of solutions and corresponding numerical methods discussed in the literature.

Finite difference methods have been proposed for \cref{eq:main} upon discretizing the Hessian $\nabla^2 u$ directly and assuming continuity of $A$. These methods are known to require wide stencils to enforce monotonicity and consistency, the required size of those stencils depending on the ellipticity constant of $A$\cite{motzkin1953,kocan1995}. We refer to the semi-Lagrangian methods for linear and nonlinear elliptic problems by K.~Debrabant and E.~R.~Jakobsen\cite{debrabant2013} and by F.~Camilli and M.~Falcone\cite{camilli1995}, which are two-scale methods. We also refer to the methods by J.~F.~Bonnans, É.~Ottenwaelter, and H.~Zidani\cite{bonnans2004} for $d=2$ and J.~F.~Bonnans, G.~Bonnet, and J.-M.~Mirebeau\cite{bonnans2023monotone} for $d>2$, which feature more compact stencils, to the extent permitted by the ellipticity of $A$, at the cost of requiring a Cartesian structure of the discretization grid. On the other hand, a two-scale Galerkin method has been proposed by R.~H.~Nochetto and W.~Zhang\cite{nochetto2018}. All these methods are monotone and are known to converge in the max-norm.

An essential difficulty for Galerkin discretizations of \cref{eq:main} is the lack of a natural variational formulation associated to this problem.
If the matrix field $A$ is differentiable, we clearly have that the problem \cref{eq:main} is equivalent to the divergence form problem
\begin{equation*}
    \begin{cases}
        \ddiv (A \nabla u) - (\ddiv A) \cdot \nabla u = 0 &\text{in } \Omega, \\
        u = g &\text{on } \partial \Omega,
    \end{cases}
\end{equation*}
which leads to the usual variational formulation. However, this approach is \emph{not} applicable when either $A$ is not differentiable or $A$ is differentiable but rough, thereby giving rise to an advection-dominated diffusion  problem. Moreover, in typical applications arising from fully nonlinear PDEs, the coefficients cannot be expected to be continuous everywhere and the location of discontinuities cannot be assumed to be known a priori. Thus, the discretization of \cref{eq:main} without any further assumption on the regularity of the coefficients is of great importance.

An approach introduced in Ref.~\refcite{smears2013} for solving numerically the problem \cref{eq:main}, and that is applicable when the matrix field $A$ is rough and satisfies the Cordes condition, consists in discretizing the following variational formulation:  find $u \in V^g$ such that
\begin{equation}
    \label{eq:var_explicit}
    \int_\Omega (\gamma A : \nabla^2 u) \Delta v\, d x = \int_\Omega \gamma f \Delta v\, d x, \quad \forall v \in V^0,
\end{equation}
where
\begin{align*}
    V^g &:= \{u \in H^2(\Omega) \mid u|_{\partial \Omega} = g|_{\partial \Omega}\},
\end{align*}
and where $\gamma \in L^\infty(\Omega; \RR_+^*)$ is a suitable scaling function.
Since the range of the Laplace operator $\Delta:H^2(\Omega)\to L^2(\Omega)$ is the whole of $L^2(\Omega)$, because of the (crucial) convexity assumption on $\Omega$, and the boundary condition is just of Dirichlet type for $u$, the variational formulation \cref{eq:var_explicit} is equivalent to the \emph{strong form} second order equation \cref{eq:main}.

%At first sight, problem \cref{eq:var_explicit} appears as a variational formulation of a forth-order problem. Notice, however, that only one boundary condition is used  and  the (crucial) convexity assumption on $\Omega$ implies that the range of the Laplace operator is the whole of $L^2(\Omega)$, thus the variational formulation \cref{eq:var_explicit} is, in fact, equivalent to the \emph{strong form} second order equation \cref{eq:main}.

It can be shown that \cref{eq:var_explicit} is well-posed in $V$ in the setting of the Cordes conditions. The proof  hinges upon the Miranda-Talenti estimate; see Ref.~\refcite{smears2013} and \cref{thm:miranda_talenti} below.
Hence, the variational formulation \cref{eq:var_explicit} allows for the direct finite element discretization of strong solutions via $H^2$-conforming elements\cite{smears2013,gallistl2017}.
For instance,  the $C^1$-conforming Bogner-Fox-Schmit (BFS) is shown by D.~Gallistl in Ref.~\refcite{gallistl2017} to produce second order accurate solutions over rectangular partitions for $d=2$.
But,  in general, the need for $H^2$-conforming elements, which  are notoriously cumbersome to construct and implement, makes the straightforward discretization of \cref{eq:var_explicit} challenging or even impractical when $d\ge 3$.
%The main difficulty that arises when attempting to discretize \cref{eq:var_explicit} is that $H^2$-conformity, which would be a natural requirement for the discretization, tends to be impractical to achieve using standard finite element methods.
%

Earlier,  I.~Smears and E.~Süli introduced in Ref.~\refcite{smears2013} the alternative idea of relaxing the conformity requirements by considering non-conforming discrete variational formulations including stabilization terms. Their formulation is designed so that stability can be shown  based on establishing a discrete analogue of the Miranda-Talenti estimate whose proof is highly technical.
 %Alternatively, the conformity requirements can be relaxed by considering discrete variational formulations designed to restore stability and hence provide well-posed discretisations.
%
Later approaches essentially follow Ref.~\refcite{smears2013}, proposing non-conforming stabilized discretizations relying on discrete Miranda-Talenti estimates; see e.g. Refs.~\refcite{dedner2022nonvar,gallistl2017,gallistl2022,kawecki2021,neilan2019,wu2021}.

In contrast, in this paper, we propose and study the discretization of \cref{eq:var_explicit} using an $H^2$-conforming  VEM. %, which has been shown to provide a framework for the design of $H^2$-conforming numerical schemes with a high degree of conformity.
Originally, the VEM was proposed in Ref.~\refcite{beirao2013} as a generalization of the $H^1$-conforming FEM to polygonal/polyhedral meshes. However, it was soon realized that the flexibility of its design could be exploited to construct entirely new discrete spaces satisfying specific properties, such as a given type of conformity. %, such as conservation of physically relevant quantities, symmetries, and conformity.
An $H^2$-conforming  VEM was already proposed by F.~Brezzi and D.~Marini in Ref.~\refcite{brezzi2013} for the solution of plate bending problems and has since been generalized by several authors including for the solution of polyharmonic problems\cite{beirao2014,brenner2019,antonietti2016,antonietti2020,beirao2020,antonietti2021,dedner2022fourthorder,chen2022hessian,antonietti2022}. Most recently, based on a hierarchical construction in the space dimension, a complete family of virtual element spaces  satisfying any degree of conformity  in arbitrary dimensions has been presented by C.~Chen, X.~Huang, and H.~Wei in Ref.~\refcite{chen2022conforming}.

In terms of both their design and implementation,  VEMs with higher degree of conformity are conceptually similar to the basic $H^1$-conforming VEM. This stands in stark contrast with their finite elements counterparts, which are characterized by ad hoc constructions leading to spaces of relatively large dimensions that are \emph{not} affine equivalent and, hence, cumbersome to implement resorting on a reference element as standard\cite{ciarlet2002,kirby2019}. For this reason, VEMs, in which elements are also constructed on the physical domain directly, without the use of reference-to-physical mappings, are competitive in terms of implementation difficulty and cost against finite elements with high degrees of conformity, while also providing some benefits such as a low number of degrees of freedom and a consistent construction of the relevant virtual element spaces regardless of the dimension and the required polynomial consistency order.

Virtual elements are of generalized finite element type in that they include both polynomial \emph{and} non-polynomial functions. For instance, the $H^2$-conforming virtual element space of (optimal) order $m-1$ in the $H^2$-seminorm  considered herein contains the space $\PP_m$ of polynomials of degree up to $m$ but no polynomial of higher degree. The space is completed by non-polynomial functions in view of the conformity requirements. The dimension of the resulting element is inherently consistent with the required level of  conformity and approximation power\cite{chen2022conforming}.
%On the other hand, virtual basis functions have an implicit definition and as such they can only be accessed through their degrees of freedom. In turns, the VEM implementation relies on the availability of  certain projections onto $\PP_m$ to be computable only from the degrees of freedom....

%Hereafter we adopt such a hierarchical construction of $H^2$-conforming virtual element spaces and analyze their performance for the solution of \cref{eq:var_explicit}.
As far as we are aware, this is the first application of the VEM to elliptic problems in nondivergence form. Our approach yields a rather compact sparsity pattern, compared with two-scale methods, but at the expense of monotonicity. Therefore, our notions of stability and convergence are not expressed in $L^\infty$-norm but rather in the $H^2$-norm.

We show that the $H^2$-conforming VEM discretization of  the prototype elliptic equation in nondivergence form~\cref{eq:var_explicit} is well-posed by directly exploiting the continuous Miranda-Talenti estimate. Furthermore, we prove stability and optimal rates of convergence in the $H^2$-norm, including the effect of quadrature. The method and analysis apply to both standard and polygonal/polyhedral meshes, under a common shape-regularity assumption, namely uniformly star-shapedness of both the elements and their interfaces. This assumption reduces to classical shape-regularity in the case of standard meshes made of simplices.

Crucially, the analysis holds under no additional regularity on the data $(A,f,g)$, hence in exactly the same basic setting for existence of \cref{eq:main}. This is in contrast to error estimates for VEMs applied to elliptic problems in divergence form, which typically involve regularity of the coefficients, see e.g. Ref.~\refcite{cangiani2017}, as is also customary for FEMs\cite{ciarlet2002}. The reason is that we are able to exploit the strong form \cref{eq:main} in the analysis; see Ref.~\refcite{kawecki2021} for a similar argument in the setting of discontinuous Galerkin methods.
In view of showing optimal convergence under minimal regularity, we also construct and analyze a
Scott-Zhang\cite{scott1990} type interpolation operator for the $H^2$-conforming VEM spaces of Ref.~\refcite{chen2022conforming}.
Furthermore, the analysis is extended to include the effect of quadrature under some minimal continuity assumptions  required for point-wise evaluation of the data.
These results are verified in practice with a set of numerical examples carefully chosen to examine the effect of the regularity of coefficients and solution on the convergence rate. For instance, they confirm that optimal convergence rates are preserved when the coefficients are rough but the solution is smooth.%, thereby corroborating our theoretical predictions.

$H^2$-conforming VEMs are conceptually simple and greatly simplify the analysis upon circumventing the need of a discrete Miranda-Talenti estimate. In addition, they may have other potential advantages. For instance, the fact that (high order) approximations of the gradient are readily available from the numerical solutions is of particular interest in optimal control applications.  In contrast, state-of-the-art approaches based on finite difference methods typically provide low order gradient approximations\cite{bonnans2004,bonnans2023monotone,camilli1995,debrabant2013}. We further highlight that the degrees of freedom of the lowest order $H^2$-conforming VEM, namely the value and gradient at the mesh vertices, are especially simple. This is true in all dimensions, making the extension of the method to moderately  higher-dimensional problems plausible. Finally, the fact that we deal directly with the strong form \cref{eq:main} entails immediate access to residual error estimates, as in Ref.~\refcite{gallistl2017}. Such estimates may be used to drive automatic mesh adaptivity, thereby exploiting the mesh flexibility associated with VEMs.

The rest of this paper is organized as follows. In \cref{sec:cont}, we discuss the Cordes condition and show unique solvability of \cref{eq:main}. In \cref{sec:discr}, we introduce an $H^2$-conforming VEM framework in arbitrary dimensions for which we show existence of a unique discrete solution, and prove a quasi-optimal error estimate in the $H^2(\Omega)$-norm. In \cref{sec:quadrature}, we prove a variant of the error estimate taking into account the effect of numerical quadrature on the scheme. In \cref{sec:h2vem} we recall the construction of  $H^2$-conforming VEMs satisfying the framework of our analysis. For simplicity, we limit ourselves to the case $d=2,3$ and analyze the corresponding quasi-interpolation operator. We use the latter to derive optimal convergence rates in \cref{sec:h2vem_rates}. We finally conclude in \cref{sec:numerics} with a set of numerical experiments exploring the interplay of regularity of solution and data with the accuracy of the VEM.

%%%%%%%%%%%%%%%%%%%%%%%%%%%%%%%%%%%%%%%%%%%%%%%%%%%%%%%%%%%%%%%%%%%%%%%%%%%%%%%%%%%%%%
\section{The continuous problem}
\label{sec:cont}
%%%%%%%%%%%%%%%%%%%%%%%%%%%%%%%%%%%%%%%%%%%%%%%%%%%%%%%%%%%%%%%%%%%%%%%%%%%%%%%%%%%%%%

In this section, we review some elements of the analysis of $H^2$ solutions to equation \cref{eq:main}. We start by defining the Cordes condition in \cref{subsec:cordes}, and then in \cref{subsec:var} we discuss the variational formulation \cref{eq:var_explicit} and its well-posedness. The proof of the well-posedness of \cref{eq:var_explicit} under the Cordes condition is standard\cite{maugeri2000,smears2013}; we choose to reproduce it for completeness, and in order to make it easy to compare with the analysis of our discrete variational formulation in \cref{subsec:scheme_well_posedness}.

\subsection{The Cordes condition}
\label{subsec:cordes}

We equip $\RR^{d \times d}$ with the Frobenius norm, that we denote by $|\cdot|$.

\begin{definition}[Cordes condition]
    \label{def:cordes}
    Let $0 \leq \mu < 1$. A symmetric positive definite matrix $\bA \in \RR^{d \times d}$ is said to satisfy the \emph{$\mu$-Cordes condition} if there exists $0 \leq \mu < 1$ such that
    \begin{equation*}
        \frac{|\bA|^2}{(\tr \bA)^2} \leq \frac{1}{d - \mu^2}.
    \end{equation*}
    The matrix field $A \in [L^\infty(\Omega)]^{d \times d}$ is said to satisfy the \emph{$\mu$-Cordes condition} if $A(x)$ satisfies the $\mu$-Cordes condition for almost every $x \in \Omega$.
\end{definition}

\begin{remark}
    The Cordes condition is often written as $|\bA|^2 / (\tr \bA)^2 \leq 1 / (d - 1 + \epsilon)$ for some $0 < \epsilon \leq 1$. This is equivalent to the definition above, with $\mu = \sqrt{1 - \epsilon}$.
\end{remark}

\begin{proposition}[characterizations of the Cordes condition]
    \label{prop:cordes_charac}
    Let $0 \leq \mu < 1$, and let $\bA \in \RR^{d \times d}$ be symmetric positive definite. The following are equivalent:
    \begin{romanlist}[(i)]
        \item $\bA$ satisfies the $\mu$-Cordes condition.
        \item There exists $\bgamma > 0$ such that $|\bgamma \bA - I_d| \leq \mu$.
        \item One has $|\bgamma \bA - I_d| \leq \mu$ with $\bgamma = \tr \bA / |\bA|^2$.
    \end{romanlist}
\end{proposition}

\begin{proof}
    For any $\bgamma \in \RR$, one has
    \begin{equation}
        \label{eq:cordes_charac_development}
        |\bgamma \bA - I_d|^2 = \bgamma^2 |\bA|^2 - 2 \bgamma \tr \bA + d.
    \end{equation}
    In particular, if $\bgamma = \tr \bA / |\bA|^2$, then $|\bgamma \bA - I_d|^2 = d - (\tr \bA)^2 / |\bA|^2$, from which the equivalence between (i) and (iii) follows. It is obvious that (iii) implies (ii), and to prove that (ii) implies (iii) it suffices to prove that
    \begin{equation*}
        \tr \bA / |\bA|^2 = \argmin_{\bgamma > 0} |\bgamma \bA - I_d|^2,
    \end{equation*}
    which follows from using \cref{eq:cordes_charac_development} and then writing the first-order optimality condition for the minimization problem.
\end{proof}

The above characterization of the Cordes condition motivates the following definition:

\begin{definition}[admissible scaling]
    \label{def:admissible_scaling}
    Assume that the matrix field $A \in [L^\infty(\Omega)]^{d \times d}$ satisfies the $\mu$-Cordes condition for some $0 \leq \mu < 1$. The function $\gamma \in L^\infty(\Omega; \RR_+^*)$ is said to be a \emph{$\mu$-admissible scaling} of $A$ if $|\gamma(x) A(x) - I_d| \leq \mu$ for almost every $x \in \Omega$.
\end{definition}

By \cref{prop:cordes_charac}, if $A$ satisfies the $\mu$-Cordes condition for some $0 \leq \mu < 1$, then there exists a $\mu$-admissible scaling of $A$, and a possible choice is given by the formula $\gamma(x) = \tr A(x) / |A(x)|^2$ (note that in this case the fact that $\gamma \in L^\infty(\Omega; \RR_+^*)$ follows from the assumption that $A$ is uniformly elliptic and bounded).

The following proposition will be useful in order to simplify our estimates in the following sections.

\begin{proposition}
    \label{prop:cordes_scaled_bound}
    Assume that $A$ satisfies the $\mu$-Cordes condition for some $0 \leq \mu < 1$ and that $\gamma$ is a $\mu$-admissible scaling of $A$. Then, for any measurable set $K \subset \Omega$,
    \begin{equation*}
        \|\gamma A\|_{0, \infty, K} < 1+\sqrt{d}.%\lesssim 1.
    \end{equation*}
\end{proposition}

\begin{proof}
    For almost every $x \in K$,
    \begin{equation*}
        |\gamma(x) A(x)|
        = |\gamma(x) A(x) - I_d + I_d|
        \leq |\gamma(x) A(x) - I_d| + |I_d|
        \leq \mu + \sqrt{d} < 1+\sqrt{d},
    \end{equation*}
    which concludes the proof.
\end{proof}

%\section{The continuous problem}
%\label{sec:cont}

\subsection{Variational formulation and well-posedness}
\label{subsec:var}

Let us define the bilinear form $a \colon H^2(\Omega) \times H^2(\Omega) \to \RR$ and the linear form $L \colon H^2(\Omega) \to \RR$ by
\begin{align}
    \label{eq:bilin_lin}
    a(u, v) &:= \int_\Omega (\gamma A : \nabla^2 u) \Delta v\, d x, &
    L(v) &:= \int_\Omega \gamma f \Delta v\, d x,
\end{align}
so that the variational problem \cref{eq:var_explicit} can be written as: find $u \in V^g$ such that
\begin{equation}
    \label{eq:var}
    a(u, v) = L(v), \quad \forall v \in V^0.
\end{equation}
We emphasize that the variational formulation \cref{eq:var} is consistent with the problem \cref{eq:main}.

\begin{proposition}[consistency]
    \label{prop:consistency}
    For any $u \in V^g$, the following are equivalent:
    \begin{romanlist}[(i)]
        \item $A : \nabla^2 u = f$ almost everywhere in $\Omega$.
        \item $\gamma A : \nabla^2 u = \gamma f$ almost everywhere in $\Omega$.
        \item $u$ is solution to \cref{eq:var}.
    \end{romanlist}
\end{proposition}

\begin{proof}
    The equivalence between (i) and (ii) follows from the fact that $\gamma > 0$ almost everywhere in $\Omega$. Since $\gamma A : \nabla^2 u \in L^2(\Omega)$ and $\gamma f \in L^2(\Omega)$, (ii) is satisfied if and only if
    \begin{equation*}
        \int_\Omega (\gamma A : \nabla^2 u) \phi\, d x = \int_\Omega \gamma f \phi\, d x, \quad \forall \phi \in L^2(\Omega).
    \end{equation*}
    Using that $\Delta$ is a bijection from $V^0$ to $L^2(\Omega)$, we deduce that (ii) is equivalent to (iii).
\end{proof}

Let us now explain how one can use the Lax-Milgram lemma in order to prove the well-posedness of \cref{eq:var}, provided that the Cordes condition is satisfied and that the scaling function $\gamma$ is chosen appropriately. We start by stating the continuity of $a$ and $L$.

\begin{proposition}[continuity]
    \label{prop:continuity}
    For  any $u$, $v \in H^2(\Omega)$, one has
    \begin{align*}
        a(u, v) &\lesssim \|\gamma A\|_{0, \infty} |u|_{2} |v|_{2}, &
        L(v) &\lesssim \|\gamma f\|_{0} |v|_{2},
    \end{align*}
with hidden constants only depending on $d$.
\end{proposition}

\begin{proof}
    This is immediately verified.
\end{proof}

The proof of the coercivity of $a$ relies on the \emph{Miranda-Talenti estimate}, which we recall below.

\begin{theorem}[Miranda-Talenti estimate]
    \label{thm:miranda_talenti}
    For any $u \in V^0$, one has $\|\nabla^2 u\|_0 \leq \|\Delta u\|_0$.
\end{theorem}

\begin{proof}
    See Refs.~\refcite{maugeri2000,smears2013}. Note that the proof uses the facts that $\Omega$ is convex and that $u|_{\partial \Omega} = 0$.
\end{proof}

\begin{proposition}[coercivity]
    \label{prop:coercivity}
    Assume that $A$ satisfies the $\mu$-Cordes condition for some $0 \leq \mu < 1$ and that $\gamma$ is a $\mu$-admissible scaling of $A$. Then, for any $u \in V^0$, one has $a(u, u) \geq (1 - \mu) |u|_2^2$.
\end{proposition}

\begin{proof}
    One has
    \begin{align*}
        a(u, u)
        &= \int_\Omega (\gamma A : \nabla^2 u) \Delta u\, d x
        = \|\Delta u\|_0^2 + \int_\Omega ((\gamma A - I_d) : \nabla^2 u) \Delta u\, d x \\
        &\geq \|\Delta u\|_0^2 - \|\gamma A - I_d\|_{0, \infty} \|\nabla^2 u\|_0 \|\Delta u\|_0
        \geq \|\Delta u\|_0^2 - \mu \|\nabla^2 u\|_0 \|\Delta u\|_0 \\
        &\geq (1 - \mu) \|\Delta u\|_0^2,
    \end{align*}
    where we used that $\gamma$ is a $\mu$-admissible scaling of $A$ (see \cref{def:admissible_scaling}) for the second inequality and the Miranda-Talenti estimate (\cref{thm:miranda_talenti}) for the last inequality. We conclude by invoking once again the Miranda-Talenti estimate.
\end{proof}

\begin{remark}[Poincaré-type inequality]
    \label{rem:poincare}
    The $H^2$ seminorm $|\cdot|_2$ is a norm on $V^0$, and for any $u \in V^0$, one has $\|u\|_2 \lesssim |u|_2$, with a hidden constant depending only on $d$ and $\Omega$. One way to prove this inequality is using the a priori estimate for the Poisson equation, according to which $\|u\|_2 \lesssim \|\Delta u\|_0$, and then the easily verified inequality $\|\Delta u\|_0 \leq \sqrt{d} \|\nabla^2 u\|_0 = \sqrt{d} |u|_2$.
\end{remark}

\begin{corollary}[well-posedness]
    \label{coro:well_posedness}
    Assume that $A$ satisfies the $\mu$-Cordes condition for some $0 \leq \mu < 1$ and that $\gamma$ is a $\mu$-admissible scaling of $A$. Then there exists a unique solution $u \in V^g$ to \cref{eq:var}.
\end{corollary}

\begin{proof}
    Let $u^g \in V^g$. Then $u$ is solution \cref{eq:var} if and only if $u = u^g + u^0$, where $u^0$ is solution to: find $u^0 \in V^0$ such that
    \begin{equation*}
        a(u^0, v) = L(v) - a(u^g, v), \quad \forall v \in V^0.
    \end{equation*}
    The existence of a unique solution $u^0$ follows from \cref{prop:continuity,prop:coercivity} and the Lax-Milgram lemma.
\end{proof}

%%%%%%%%%%%%%%%%%%%%%%%%%%%%%%%%%%%%%%%%%%%%%%%%%%%%%%%%%%%%%%%%%%%%%%%%%%%%%%%%%%%%%%
\section{Virtual element discretization}
\label{sec:discr}
%%%%%%%%%%%%%%%%%%%%%%%%%%%%%%%%%%%%%%%%%%%%%%%%%%%%%%%%%%%%%%%%%%%%%%%%%%%%%%%%%%%%%%

%Recall that $\Omega$ is a bounded convex polytopal domain. For any subset $K \subset \Omega$, we denote by $h_K$ the diameter of $K$.

In this section, we introduce and discuss the main properties of an $H^2$-conforming VEM based on the variational formulation \cref{eq:var}. After defining the scheme in \cref{subsec:scheme}, we prove its well-posedness in \cref{subsec:scheme_well_posedness} and derive an error estimate in \cref{subsec:error_estimate}.

\subsection{Description of the scheme}
\label{subsec:scheme}

Let us first introduce the setting of the discretization. We assume that the following parameters are given:
\begin{itemlist}
    \item A partition $\cT_h$ of $\Omega$ in finitely many nonoverlapping polytopes.
%\todo{Assuming openness is restrictive for \cref{sec:quadrature}, but not assuming it would require introducing more complicated notations in \cref{sec:quadrature}. I do not know what is best.}
    \item A polynomial consistency order $m \geq 2$ for the numerical scheme.
\end{itemlist}
Moreover,  we assume that $\cT_h$ satisfies the following standard shape-regularity properties:
\begin{itemlist}
    \item There exists $\rho > 0$ such that any $K \in \cT_h$ and any $d'$-dimensional facet $F$ of $K$, with $1 \leq d' \leq d$, are star-shaped with respect to respectively a $d$-dimensional ball of radius $\rho h_K$ and a $d'$-dimensional ball of radius $\rho h_F$.
    \item There exists $\eta > 0$ such that any $d'$-dimensional facet $F$ of any $K \in \cT_h$, with $1 \leq d' \leq d$, satisfies $h_F \geq \eta h_K$.
\end{itemlist}
Here $h_\omega$ denotes the diameter of $\omega \subset \Omega$. We further define $h := \max_{K \in \cT_h} h_K$.

From now on, for any two quantities $a$, $b \in \RR$, we write $a \lesssim b$ if $a \leq C b$ with some constant $C > 0$ depending only on $d$, $\Omega$, $\rho$, $\eta$, and $m$.

For $k \in \NN$ and $K \in \cT_h$, we denote by $\PP_k(K)$ the space of polynomials of degree $k$ on $K$, and by $\PP_k(\cT_h)$ the set of functions $p \colon \Omega \to \RR$ whose restrictions $p|_K$ to each $K \in \cT_h$ belong to $\PP_k(K)$.
We denote by $\Pi_k^0 \colon L^2(\Omega) \to \PP_k(\cT_h)$ (respectively $L^2(K) \to \PP_k(K)$) the $L^2$ projection operator onto $\PP_k(\cT_h)$ (respectively onto $\PP_k(K)$, when applied to a function defined only on some $K \in \cT_h$). We naturally extend the definition of the projection operator $\Pi_k^0$ to matrix-valued functions.

It is also useful to define as follows, for any $K \subset \cT_h$, the local bilinear and linear forms $a_K \colon H^2(K) \times H^2(K) \to \RR$ and $L_K \colon H^2(K) \to \RR$:
\begin{align*}
    a_K(u, v) &:= \int_K (\gamma A : \nabla^2 u) \Delta v\, d x, &
    L_K(v) &:= \int_K \gamma f \Delta v\, d x.
\end{align*}

There are several variants of the construction of $H^2$-conforming virtual element methods. For completeness, we describe a possible construction in \cref{sec:h2vem}. For now, we abstract over some aspects of the construction by assuming that, on any cell $K \in \cT_h$, we are given the following (remember that $m \geq 2$ denotes an arbitrary polynomial consistency order):
\begin{itemlist}
    \item A local virtual element space $V_{h, K}$, which is a finite-dimensional subspace $V_{h, K}$ of $H^2(K)$ satisfying $\PP_m(K) \subset V_{h, K}$.
    \item A linear projection operator $\Pi_m^* \colon V_{h, K} \to \PP_m(K)$, which must be stable in the sense that, for any $u \in V_{h, K}$,
    \begin{equation}
        \label{eq:proj_stability}
        \|\Pi_m^* u\|_{2, K} \leq c_\pi \|u\|_{2, K},
    \end{equation}
    with a constant $c_\pi \geq 1$ not depending on $u$ and $K$.
    \item A stabilization form $s_{h, K} \colon V_{h, K} \times V_{h, K} \to \RR$, which must be a bilinear form satisfying, for any $u$, $v \in V_{h, K}$,
    \begin{equation}
        \label{eq:stab_scaling}
        \begin{split}
            s_{h, K}(u - \Pi_m^* u, u - \Pi_m^* u) &\geq c_* \|\Delta (u - \Pi_m^* u)\|_{0, K}^2, \\
            s_{h, K}(u - \Pi_m^* u, v - \Pi_m^* v) &\leq c^* |u|_{2, K} |v|_{2, K},
        \end{split}
    \end{equation}
    with constants $0 < c_* \leq 1 \leq c^*$ independent of $u$, $v$, and $K$.
\end{itemlist}
(Note that inequalities $c_\pi \geq 1$, $c_* \leq 1$, and $c^* \geq 1$ are not restrictive; the reason we assume them is to make the expression of our error estimates simpler.)

We shall show that the above virtual elements ansatz is sufficient to complete the analysis, postponing to~\cref{sec:h2vem} a complete description of a possible realization of the above framework.

Following \cref{eq:bilin_lin} and the usual approach for constructing virtual element methods, we define, for any $K \in \cT_h$, the bilinear and linear forms $a_{h, K} \colon V_{h, K} \times V_{h, K} \to \RR$ and $L_{h, K} \colon V_{h, K} \to \RR$ by
\begin{align*}
    a_{h, K}(u, v) &:= \int_K (\gamma A : \Pi_{m-2}^0 \nabla^2 u) \Pi_{m-2}^0 \Delta v\, d x + s_{h, K}(u - \Pi_m^* u, v - \Pi_m^* v), \\
    L_{h, K}(v) &:= \int_K \gamma f \Pi_{m-2}^0 \Delta v\, d x.
\end{align*}

We define global virtual element spaces $V_h \subset H^2(\Omega)$ and $V_h^0$, $V_h^{g_I} \subset V_h$ by
\begin{equation}
    \label{eq:vemspace}
    V_h := \{u \in H^2(\Omega) \mid u|_K \in V_{h, K},\, \forall K \in \cT_h\},
\end{equation}
and
\begin{align}
    \label{eq:vemspacedirichlet}
    V_h^0 &:= \{u \in V_h \mid u|_{\partial \Omega} = 0\}, &
    V_h^{g_I} &:= \{u \in V_h \mid u|_{\partial \Omega} = g_I|_{\partial \Omega}\},
\end{align}
where $g_I \in V_h$ is a given function. In order for $V_h^{g_I}$ to be a good approximation of the space $V^g$, the boundary values of $g_I$ should be chosen as some interpolation of those of $g$; an example of a suitable interpolation operator is described in \cref{subsec:h2vem_global_interp}.

We define the global counterparts $a_h \colon V_h \times V_h \to \RR$ and $L_h \colon V_h \to \RR$ to $a_{h, K}$ and $L_{h, K}$ by
\begin{align*}
    a_h(u, v) &:= \sum_{K \in \cT_h} a_{h, K}(u, v) = \int_\Omega (\gamma A : \Pi_{m-2}^0 \nabla^2 u) \Pi_{m-2}^0 \Delta v\, d x
    \iftoggle{author}{}{\\ &\qquad \qquad \qquad \qquad \qquad}
    + \sum_{K \in \cT_h} s_{h, K}(u - \Pi_m^* u, v - \Pi_m^* v), \\
    L_h(v) &:= \sum_{K \in \cT_h} L_{h, K}(v) = \int_\Omega \gamma f \Pi_{m-2}^0 \Delta v\, d x,
\end{align*}
where the projection operator $\Pi_m^* \colon V_h \to \PP_m(\cT_h)$ is the natural extension to $V_h$ of the local projection operators $\Pi_m^* \colon V_{h, K} \to \PP_m(K)$.

The virtual element scheme that we study is the following: find $u_h \in V_h^{g_I}$ such that
\begin{equation}
    \label{eq:scheme}
    a_h(u_h, v) = L_h(v), \quad \forall v \in V_h^0.
\end{equation}

\subsection{Existence of a unique discrete solution}
\label{subsec:scheme_well_posedness}

In order to show the well-posedness of the scheme \cref{eq:scheme}, we adapt to the discrete setting the arguments used in \cref{subsec:var} in the case of the continuous problem \cref{eq:var}: \cref{coro:discr_continuity}, \cref{prop:discr_coercivity}, and \cref{coro:discr_well_posedness} below are counterparts in the discrete setting to, respectively, \cref{prop:continuity}, \cref{prop:coercivity}, and \cref{coro:well_posedness}.

\begin{proposition}[continuity, cell-level]
    \label{prop:discr_continuity_cell}
    For any $K \in \cT_h$ and any $u$, $v \in V_{h, K}$, one has
    \begin{align*}
        a_{h, K}(u, v) &\lesssim (\|\gamma A\|_{0, \infty, K} + c^*) |u|_{2, K} |v|_{2, K}, &
        L_{h, K}(v) &\lesssim \|\gamma f\|_{0, K} |v|_{2, K},
    \end{align*}
    where $c^*$ is from \cref{eq:stab_scaling}.
\end{proposition}

\begin{proof}
    This is immediately verified.
\end{proof}

\begin{corollary}[continuity, global level]
    \label{coro:discr_continuity}
    For any $u$, $v \in V_h$, one has
    \begin{align*}
        a_h(u, v) &\lesssim (\|\gamma A\|_{0, \infty} + c^*) |u|_2 |v|_2, &
        L_h(v) &\lesssim \|\gamma f\|_0 |v|_2,
    \end{align*}
    where $c^*$ is from \cref{eq:stab_scaling}.
\end{corollary}

\begin{proof}
    This follows from \cref{prop:discr_continuity_cell} and the Cauchy-Schwarz inequality.
\end{proof}

\begin{proposition}[coercivity]
    \label{prop:discr_coercivity}
    Assume that $A$ satisfies the $\mu$-Cordes condition for some $0 \leq \mu < 1$ and that $\gamma$ is a $\mu$-admissible scaling of $A$. Then for any $u \in V_h^0$, one has $a_h(u, u) \geq (c_* - \mu) |u|_2^2$, where $c_*$ is from \cref{eq:stab_scaling}.
\end{proposition}

\begin{proof}
    Recall that
    \begin{equation*}
        a_h(u, u) = \int_\Omega (\gamma A : \Pi_{m-2}^0 \nabla^2 u) \Pi_{m-2}^0 \Delta u\, d x + \sum_{K \in \cT_h} s_{h, K}(u - \Pi_m^* u, u - \Pi_m^* u).
    \end{equation*}
    One has
    \begin{align*}
        \int_\Omega (\gamma A : \Pi_{m-2}^0 \nabla^2 u) \Pi_{m-2}^0 \Delta u\, d x
        &= \|\Pi_{m-2}^0 \Delta u\|_0^2 \\
        &\quad + \int_\Omega ((\gamma A - I_d) : \Pi_{m-2}^0 \nabla^2 u) \Pi_{m-2}^0 \Delta u\, d x \\
        &\geq \|\Pi_{m-2}^0 \Delta u\|_0^2 \\
        &\quad - \|\gamma A - I_d\|_{0, \infty} \|\Pi_{m-2}^0 \nabla^2 u\|_0 \|\Pi_{m-2}^0 \Delta u\|_0 \\
        &\geq \|\Pi_{m-2}^0 \Delta u\|_0^2 - \|\gamma A - I_d\|_{0, \infty} \|\nabla^2 u\|_0 \|\Delta u\|_0 \\
        &\geq \|\Pi_{m-2}^0 \Delta u\|_0^2 - \mu \|\nabla^2 u\|_0 \|\Delta u\|_0 \\
        &\geq \|\Pi_{m-2}^0 \Delta u\|_0^2 - \mu \|\Delta u\|_0^2,
    \end{align*}
    where we used that $\gamma$ is a $\mu$-admissible scaling of $A$ (see \cref{def:admissible_scaling}) for the third inequality and the Miranda-Talenti estimate (\cref{thm:miranda_talenti}) for the last inequality. On the other hand,
    \begin{align*}
        \sum_{K \in \cT_h} s_{h, K}(u - \Pi_m^* u, u - \Pi_m^* u)
        &\geq c_* \sum_{K \in \cT_h} \|\Delta u - \Delta \Pi_m^* u\|_{0, K}^2 \\
        &\geq c_* \sum_{K \in \cT_h} \|\Delta u - \Pi_{m-2}^0 \Delta u\|_{0, K}^2 \\
        &= c_* \|\Delta u - \Pi_{m-2}^0 \Delta u\|_0^2,
    \end{align*}
    where for the second inequality we used that $\Delta \Pi_m^* u \in \PP_{m-2}(K)$ and that $\Pi_{m-2}^0 \Delta u$ is the $L^2$ projection of $\Delta u$ onto $\PP_{m-2}(K)$. Summing up the above, one has
    \begin{align*}
        a_h(u, u)
        &\geq \|\Pi_{m-2}^0 \Delta u\|_0^2 - \mu \|\Delta u\|_0^2 + c_* \|\Delta u - \Pi_{m-2}^0 \Delta u\|_0^2 \\
        &\geq c_* (\|\Pi_{m-2}^0 \Delta u\|_0^2 + \|\Delta u - \Pi_{m-2}^0 \Delta u\|_0^2) - \mu \|\Delta u\|_0^2
        = (c_* - \mu) \|\Delta u\|_0^2.
    \end{align*}
    We conclude using the Miranda-Talenti estimate.
\end{proof}

Note that the above proves the coercivity of $a_h$ on $V_h^0 \times V_h^0$ only when $c_* > \mu$. In that case, we can apply the Lax-Milgram lemma to establish the well-posedness of the VEM.

\begin{corollary}[well-posedness of the scheme]
    \label{coro:discr_well_posedness}
    Assume that $A$ satisfies the $\mu$-Cordes condition for some $0 \leq \mu < 1$, that $\gamma$ is a $\mu$-admissible scaling of $A$, and that $c_* > \mu$ in \cref{eq:stab_scaling}. Then there exists a unique solution $u_h \in V_h^{g_I}$ to \cref{eq:scheme}.
\end{corollary}

\begin{proof}
    Let $u_h^{g_I} \in V_h^{g_I}$. Then $u_h$ is solution \cref{eq:scheme} if and only if $u_h = u_h^{g_I} + u_h^0$, where $u_h^0$ is solution to: find $u_h^0 \in V_h^0$ such that
    \begin{equation*}
        a_h(u_h^0, v) = L_h(v) - a_h(u_h^{g_I}, v), \quad \forall v \in V_h^0.
    \end{equation*}
    The existence of a unique solution $u_h^0$ follows from \cref{coro:discr_continuity}, \cref{prop:discr_coercivity}, and the Lax-Milgram lemma.
\end{proof}

\begin{remark}
    The assumption \cref{eq:stab_scaling} on the scaling of $s_{h, K}$ is standard in the virtual element setting. However, while in the setting of divergence form elliptic equations it is usually sufficient to assume that $c_* > 0$, here we require $c_* > \mu$. This requirement is not restrictive. Admissible definitions of the stabilization form $s_{h, K}$ can always include some scaling factor, which can always be chosen large enough so that $c_* > \mu$. We refer to \cref{subsec:h2vem_proj_stab} for an example with more discussion.
\end{remark}

\subsection{Error estimate}
\label{subsec:error_estimate}

We now prove our main error estimate for the scheme \cref{eq:scheme}.

\begin{theorem}[error estimate]
    \label{thm:error_estimate}
    Assume that $A$ satisfies the $\mu$-Cordes condition for some $0 \leq \mu < 1$, that $\gamma$ is a $\mu$-admissible scaling of $A$, and that $c_* > \mu$ in \cref{eq:stab_scaling}. Let $u \in V^g$ and $u_h \in V_h^{g_I}$ be respectively the unique solutions to \cref{eq:var} and \cref{eq:scheme}. Then, for any $u_I \in V_h^{g_I}$,
    \begin{equation}
        \label{eq:error_estimate}
        \|u - u_h\|_2 \lesssim \frac{c^*}{c_* - \mu} \left(\|u - u_I\|_2 + |u - \Pi_m^0 u|_{2, h}\right).
    \end{equation}
\end{theorem}

\begin{proof}
    Using the triangle inequality and \cref{rem:poincare}, one has $\|u - u_h\|_2 \leq \|u - u_I\|_2 + \|u_I - u_h\|_2 \lesssim \|u - u_I\|_2 + |u_I - u_h|_2$, so we only have to estimate $|u_I - u_h|_2$. For convenience, we let $\delta_h := u_I - u_h$. Using \cref{prop:discr_coercivity} and the fact that $u$ and $u_h$ are solutions to respectively \cref{eq:var} and \cref{eq:scheme},
    \begin{align*}
        (c_* - \mu) |\delta_h|_2^2
        &\leq a_h(\delta_h, \delta_h)
        = a_h(u_I, \delta_h) - a_h(u_h, \delta_h)
        = a_h(u_I, \delta_h) - L_h(\delta_h) \\
        &= \sum_{K \in \cT_h} (a_{h, K}(u_I, \delta_h) - L_{h, K}(\delta_h)) \\
        &= \sum_{K \in \cT_h} a_{h, K}(u_I - \Pi_m^0 u, \delta_h) + \sum_{K \in \cT_h} (a_{h, K}(\Pi_m^0 u, \delta_h) - L_{h, K}(\delta_h)).
    \end{align*}
    For any $K \in \cT_h$, by \cref{prop:cordes_scaled_bound} one had $\|\gamma A\|_{0, \infty, K} \lesssim 1 \leq c^*$, thus, using \cref{prop:discr_continuity_cell} and the triangle inequality,
    \begin{equation*}
        a_{h, K}(u_I - \Pi_m^0 u, \delta_h) \lesssim c^* (|u - u_I|_{2, K} + |u - \Pi_m^0 u|_{2, K}) |\delta_h|_{2, K}.
    \end{equation*}
    To estimate the term $a_{h, K}(\Pi_m^0 u, \delta_h) - L_{h, K}(\delta_h)$, we note that
    \begin{equation*}
        s_{h, K}(\Pi_m^0 u - \Pi_m^* \Pi_m^0 u, \delta_h - \Pi_m^* \delta_h) = s_{h, K}(0, \delta_h - \Pi_m^* \delta_h) = 0,
    \end{equation*}
    so that
    \begin{equation*}
        a_{h, K}(\Pi_m^0 u, \delta_h) - L_{h, K}(\delta_h) = \int_K (\gamma A : \nabla^2 \Pi_m^0 u) \Pi_{m-2}^0 \Delta \delta_h\, d x - \int_K \gamma f \Pi_{m-2}^0 \Delta \delta_h\, d x.
    \end{equation*}
    By \cref{prop:consistency}, one has $f = A : \nabla^2 u$ almost everywhere in $\Omega$, and thus
    \begin{align*}
        a_{h, K}(\Pi_m^0 u, \delta_h) - L_{h, K}(\delta_h)
        &= \int_K (\gamma A : \nabla^2 \Pi_m^0 u) \Pi_{m-2}^0 \Delta \delta_h\, d x \\
        &\quad - \int_K (\gamma A : \nabla^2 u) \Pi_{m-2}^0 \Delta \delta_h\, d x \\
        &= \int_K (\gamma A : \nabla^2 (\Pi_m^0 u - u)) \Pi_{m-2}^0 \Delta \delta_h\, d x.
    \end{align*}
    We deduce that $a_{h, K}(\Pi_m^0 u, \delta_h) - L_{h, K}(\delta_h) \lesssim \|\gamma A : \nabla^2 (u - \Pi_m^0 u)\|_{0, K} |\delta_h|_{2, K}$, and then, using \cref{prop:cordes_scaled_bound}, that
    \begin{equation*}
        a_{h, K}(\Pi_m^0 u, \delta_h) - L_{h, K}(\delta_h) \lesssim |u - \Pi_m^0 u|_{2, K} |\delta_h|_{2, K}.
    \end{equation*}
    Collecting the above bounds we easily conclude.
\end{proof}

\begin{remark}
    The term $|u - \Pi_m^0 u|_{2, h}$ in \cref{eq:error_estimate} can be estimated using the classical Scott-Dupont theory\cite{brenner2008}, depending on the regularity of $u$. Estimation of  the term $\|u - u_I\|_2$ requires the construction of a suitable interpolant $u_I \in V_h^{g_I}$ of $u$. This depends on   the specific  definition of the local virtual element spaces $V_{h, K}$, $K \in \cT_h$. In \cref{subsec:h2vem_global_interp} we provide an instance of such spaces for which  we are able to prove optimal interpolation error bounds, and thus deduce optimal rate of convergence for the resulting method, cf. \cref{thm:convergence_rate}.
\end{remark}

\begin{remark}
    \label{rem:no_data_regularity}
    In the context of divergence form problems, error estimates for virtual element methods usually involve regularity of the coefficients and data, see for instance Theorem~6.2 in Ref.~\refcite{cangiani2017}. In contrast, the error estimate \cref{eq:error_estimate} does not involve the functions $A$, $f$, and $\gamma$. This is thanks to \cref{prop:cordes_scaled_bound} and, importantly, to the fact that no integration by parts was performed in order to construct the variational formulation \cref{eq:var}, which allowed us to directly use the almost-everywhere equality $A : \nabla^2 u = f$ in the proof of \cref{thm:error_estimate}, rather than relying on the variational formulation \cref{eq:var} of the continuous problem. Observe that a similar argument was used in Ref.~\refcite{kawecki2021} for the analysis of discontinuous Galerkin and $C^0$-interior penalty finite element methods.
\end{remark}

\begin{remark}
    \label{rem:err_to_proj}
    Typically in virtual element methods, the numerical solution $u_h \in V_h^{g_I}$ to \cref{eq:scheme} is not accessible in practice. Instead,  the projection $\Pi_m^* u_h$ is accessible (see \cref{sec:h2vem} for more details). For this reason, it makes sense to consider the error $\|u - \Pi_m^* u_h\|_2$. This error is readily controlled combining the above estimate with the following result, which we isolate in a separate proposition since it is useful in general and can also be combined with \cref{thm:quadr_error_estimate_continuous} or \cref{thm:quadr_error_estimate_discontinuous} below, rather than \cref{thm:error_estimate}:
\end{remark}

\begin{proposition}[computable approximation]
    \label{prop:err_to_proj}
    For any $u \in H^2(\Omega)$ and $u_h \in V_h$, one has
    \begin{equation*}
        \|u - \Pi_m^* u_h\|_2 \lesssim c_\pi (|u - \Pi_m^0 u|_{2, h} + \|u - u_h\|_2).
    \end{equation*}
\end{proposition}

\begin{proof}
    By the triangle inequality and then \cref{eq:proj_stability}, one has
    \begin{align*}
        \|u - \Pi_m^* u_h\|_{2, h}
        &\leq \|u - \Pi_m^0 u\|_{2, h} + \|\Pi_m^0 u - \Pi_m^* u_h\|_{2, h} \\
        &= \|u - \Pi_m^0 u\|_{2, h} + \|\Pi_m^* (\Pi_m^0 u - u_h)\|_{2, h} \\
        &\leq \|u - \Pi_m^0 u\|_{2, h} + c_\pi \|\Pi_m^0 u - u_h\|_{2, h} \\
        &\lesssim c_\pi (\|u - \Pi_m^0 u\|_{2, h} + \|u - u_h\|_2).
    \end{align*}
    By standard polynomial projection estimates\cite{scott1990}, $\|u - \Pi_m^0 u\|_{2, h} = \|(u - \Pi_m^0 u) - \Pi_m^0 (u - \Pi_m^0 u)\|_{2, h} \lesssim |u - \Pi_m^0 u|_{2, h}$, which concludes the proof.
\end{proof}

\section{The effect of numerical integration}
\label{sec:quadrature}

As any finite element method, the assembly of a virtual element method requires some form of quadrature;  see e.g. Ref.~\refcite{cangiani2017} where this issue is discussed in the case of elliptic problems in divergence form. Given that the case of irregular data is of paramount importance, it is particularly relevant to prove that the analysis detailed in the previous section extends to a variant of the method with numerical quadrature. Here, we perform such analysis under only slightly more restrictive assumptions.

Since quadrature requires pointwise evaluation, we assume that, for any $K \in \cT_h$, we are given specific Sobolev representatives of $A_K$, $f_K$, and $\gamma_K$ of $A|_{\overline K}$, $f|_{\overline K}$, and $\gamma|_{\overline K}$, so that $A_K(x)$, $f_K(x)$, and $\gamma_K(x)$ are well-defined at every $x \in \overline K$. Note that we allow $A_{K_1}(x) \neq A_{K_2}(x)$ at $x \in \partial K_1 \cap \partial K_2$, $K_1 \neq K_2$, and likewise for $f$ and $\gamma$; this is a natural way to handle data with jumps at cell interfaces.

\begin{definition}[Cordes condition everywhere]
    Let $0 \leq \mu < 1$. We say that $A$ satisfies the $\mu$-Cordes condition \emph{everywhere} if $A_K(x)$ satisfies the $\mu$-Cordes condition for \emph{every} $K \in \cT_h$ and $x \in \overline K$. We say that $\gamma$ is an \emph{everywhere} $\mu$-admissible scaling of $A$ if $\gamma_K(x) > 0$ and $|\gamma_K(x) A_K(x) - I_d| \leq \mu$ for \emph{every} $K \in \cT_h$ and $x \in \overline K$.
\end{definition}

Let us now describe the family of quadrature rules that we allow in our analysis.

\begin{assumptions}[quadrature]
    On any cell $K \in \cT_h$, we are given a finite set of quadrature points $X_K \subset \overline K$ and corresponding nonnegative quadrature weights $(\omega_K(x))_{x \in X_K} \subset \RR_+$. For $\phi \colon X_K \to \RR$, the quadrature $Q_K[\phi]$ is defined as $Q_K[\phi] := \sum_{x \in X_K} \omega_K(x) \phi(x)$, and satisfies $Q_K[\phi] = \int_K \phi\, d x$ whenever $\phi \in \PP_{2 m - 4}(\overline K)$.
\end{assumptions}

For any $K \in \cT_h$ and $u$, $v \in V_{h, K}$, we define
\begin{align*}
    a_{h, K}^Q(u, v) &:= Q_K[(\gamma_K A_K : \Pi_{m-2}^0 \nabla^2 u) \Pi_{m-2}^0 \Delta v] + s_{h, K}(u - \Pi_m^* u, v - \Pi_m^* v), \\
    L_{h, K}^Q(v) &:= Q_K [\gamma_K f_K \Pi_{m-2}^0 \Delta v].
\end{align*}
We then define the global counterparts $a_h^Q$ and $L_h^Q$ of $a_{h, K}^Q$ and $L_{h, K}^Q$ by summing over all $K \in \cT_h$ as usual, and we consider the following scheme with quadrature: find $u_h \in V_h^{g_I}$ such that
\begin{equation}
    \label{eq:quadr_scheme}
    a_h^Q(u_h, v) = L_h^Q(v), \quad \forall v \in V_h^0.
\end{equation}

The two following results are easily proved by using the Cauchy-Schwarz inequality $Q_K[\phi \psi]^2 \leq Q_K[\phi^2] Q_K[\psi^2]$, the exactness of quadrature for polynomials of degree $2 m - 4$, and arguing as in \cref{prop:discr_continuity_cell,prop:discr_coercivity}.

\begin{proposition}[continuity, with quadrature]
    \label{prop:quadr_continuity}
    For any $K \in \cT_h$ and any $u$, $v \in V_{h, K}$, one has
    \begin{align*}
        a_{h, K}^Q(u, v) &\lesssim \left(\max_{x \in X_K} |\gamma_K(x) A_K(x)| + c^*\right) |u|_{2, K} |v|_{2, K}, \\
        L_{h, K}^Q(v) &\lesssim Q_K[|\gamma_K f_K|^2]^{1/2} |v|_{2, K},
    \end{align*}
    where $c^*$ is from \cref{eq:stab_scaling}.
\end{proposition}

\begin{proposition}[coercivity, with quadrature]
    Assume that $A$ satisfies the $\mu$-Cordes condition everywhere for some $0 \leq \mu < 1$ and that $\gamma$ is an everywhere $\mu$-admissible scaling of $A$. Then for any $u \in V_h^0$, one has $a_h^Q(u, u) \geq (c_* - \mu) |u|_2^2$, where $c_*$ is from \cref{eq:stab_scaling}.
\end{proposition}

The well-posedness of the scheme \cref{eq:quadr_scheme} can then be established using the Lax-Milgram lemma similarly to \cref{coro:discr_well_posedness}. Next, we derive an error estimate. We start with the following lemma.

\begin{lemma}
    \label{lemma:quadr_partial_error_estimate}
    Assume that $A$ satisfies the $\mu$-Cordes condition everywhere for some $0 \leq \mu < 1$, that $\gamma$ is an everywhere $\mu$-admissible scaling of $A$, and that $c_* > \mu$ in \cref{eq:stab_scaling}. Let $u_h \in V_h^{g_I}$ be the unique solution to \cref{eq:quadr_scheme}. Then, for any $u \in H^2(\Omega)$ and $u_I \in V_h^{g_I}$,
    \begin{equation}
        \label{eq:quadr_partial_error_estimate}
        \begin{split}
            \|u - u_h\|_{2, h}
            &\lesssim \frac{c^*}{c_* - \mu} (\|u - u_I\|_2 + |u - \Pi_m^0 u|_{2, h}) \\
            &\quad + \frac{1}{c_* - \mu} \left(\sum_{K \in \cT_h} Q_K[|\gamma_K (f_K - A_K : \nabla^2 \Pi_m^0 u)|^2]\right)^{1/2}.
        \end{split}
    \end{equation}
\end{lemma}

\begin{proof}
    Let $\delta_h := u_I - u_h$. Arguing as in \cref{thm:error_estimate},
    \begin{align*}
        (c_* - \mu) |\delta_h|_2^2
        &\leq \sum_{K \in \cT_h} a_{h, K}^Q(u_I - \Pi_m^0 u, \delta_h) + \sum_{K \in \cT_h} (a_{h, K}^Q(\Pi_m^0 u, \delta_h) - L_{h, K}^Q(\delta_h)).
    \end{align*}
    For any $K \in \cT_h$, using \cref{prop:quadr_continuity} and the triangle inequality, and arguing as in \cref{prop:cordes_scaled_bound} to control $\gamma_K A_K$, one has
    \begin{equation*}
        a_{h, K}^Q(u_I - \Pi_m^0 u, \delta_h) \lesssim c^* (|u - u_I|_{2, K} + |u - \Pi_m^0 u|_{2, K}) |\delta_h|_{2, K}.
    \end{equation*}
    Finally, we estimate the remaining term
    \begin{align*}
        a_{h, K}^Q(\Pi_m^0 u, \delta_h) - L_{h, K}^Q(\delta_h)
        &= Q_K [(\gamma_K A_K : \nabla^2 \Pi_m^0 u - \gamma_K f_K) \Pi_{m-2}^0 \Delta \delta_h] \\
        &\leq Q_K[|\gamma_K (f_K - A_K : \nabla^2 \Pi_m^0 u)| |\Pi_{m-2}^0 \Delta \delta_h|] \\
        &\leq Q_K[|\gamma_K (f_K - A_K : \nabla^2 \Pi_m^0 u)|^2]^{1/2} Q_K[|\Pi_{m-2}^0 \Delta \delta_h|^2]^{1/2} \\
        &\lesssim Q_K[|\gamma_K (f_K - A_K : \nabla^2 \Pi_m^0 u)|^2]^{1/2} |\delta_h|_{2, K},
    \end{align*}
    from which we conclude as in \cref{thm:error_estimate}.
\end{proof}

Note that the above lemma holds for any $u \in H^2(\Omega)$. We now want to choose $u$ as the solution to \cref{eq:var} and use the equality $f_K = A_K : \nabla^2 u$ almost everywhere in order to estimate the rightmost term in \cref{eq:quadr_partial_error_estimate}. The difficulty is that this equality is not guaranteed to hold at quadrature points. A sufficient condition for it to hold is that $A_K$, $f_K$, and $\nabla^2 u$ are continuous at quadrature points. This is the setting of our next result.

\begin{theorem}[error estimate with quadrature, continuous setting]
    \label{thm:quadr_error_estimate_continuous}
    Assume that $A$ satisfies the $\mu$-Cordes condition everywhere for some $0 \leq \mu < 1$, that $\gamma$ is an everywhere $\mu$-admissible scaling of $A$, and that $c_* > \mu$ in \cref{eq:stab_scaling}. Let $u \in V^g$ and $u_h \in V_h^{g_I}$ be respectively the unique solutions to \cref{eq:var} and \cref{eq:quadr_scheme}. Assume that, for any $K \in \cT_h$ and $x \in X_K$, $A_K$, $f_K$, and $\nabla^2 u$ are continuous at $x$. Then, for any $u_I \in V_h^{g_I}$,
    \begin{equation}
        \label{eq:quadr_error_estimate_continuous}
        \begin{split}
            \|u - u_h\|_{2, h}
            &\lesssim \frac{c^*}{c_* - \mu} (\|u - u_I\|_2 + |u - \Pi_m^0 u|_{2, h}) \\
            &\quad + \frac{1}{c_* - \mu} \left(\sum_{K \in \cT_h} Q_K[|\nabla^2 (u - \Pi_m^0 u)|^2]\right)^{1/2}.
        \end{split}
    \end{equation}
\end{theorem}

\begin{proof}
    \Cref{lemma:quadr_partial_error_estimate} applies, so we only need to estimate the rightmost term in \cref{eq:quadr_partial_error_estimate}. Since the equality $f_K = A_K : \nabla^2 u$ holds at every quadrature point, one has, for any $K \in \cT_h$,
    \begin{align*}
        Q_K[|\gamma_K (f_K - A_K : \nabla^2 \Pi_m^0 u)|^2]
        &= Q_K[|\gamma_K A_K : \nabla^2 (u - \Pi_m^0 u)|^2] \\
        &\lesssim Q_K[|\nabla^2 (u - \Pi_m^0 u)|^2],
    \end{align*}
    where we argued as in \cref{prop:cordes_scaled_bound} to control $\gamma_K A_K$.
\end{proof}

It may happen that $A$ or $f$ are discontinuous at some quadrature points (see for instance our numerical experiments on odd square and cubic meshes in \cref{subsec:numerics_roughdata_roughsol,subsec:numerics_roughdata_smoothsol}). In this case, we can still obtain some error estimate provided that the data and the solution $u$ to \cref{eq:var} satisfy the following property:
\begin{equation}
    \label{eq:quadr_set_assumption}
    \begin{split}
        &\text{\textit{there exists an open set $E \subset \Omega$ such that $u \in W^{2, \infty}(E)$ and, for}} \\
        &\text{\textit{any quadrature point $x \in X_K$, $K \in \cT_h$, one has $x \in \overline {E \cap K}$, and}} \\
        &\text{\textit{$(A_K)|_{\{x\} \cup (E \cap K)}$ and $(f_K)|_{\{x\} \cup (E \cap K)}$ are continuous at $x$}}.
    \end{split}
\end{equation}
The motivation for introducing the open set $E$ is to only assume some partial continuity of $A$ and $f$ at quadrature points. The condition $u \in W^{2, \infty}(E)$ is the weaker condition on $u$ that we were able to exploit in our analysis; while it holds for instance in all the cases we considered in our numerical experiments, we are unfortunately not aware of a general sufficient condition on the data that would guarantee this regularity of the solution.

\begin{theorem}[error estimate with quadrature, discontinuous setting]
    \label{thm:quadr_error_estimate_discontinuous}
    Assume that $A$ satisfies the $\mu$-Cordes condition everywhere for some $0 \leq \mu < 1$, that $\gamma$ is an everywhere $\mu$-admissible scaling of $A$, and that $c_* > \mu$ in \cref{eq:stab_scaling}. Let $u \in V^g$ and $u_h \in V_h^{g_I}$ be respectively the unique solutions to \cref{eq:var} and \cref{eq:quadr_scheme}. If there exists an open set $E \subset \Omega$ satisfying \cref{eq:quadr_set_assumption}, then, for any $u_I \in V_h^{g_I}$,
    \begin{equation}
        \label{eq:quadr_error_estimate_discontinuous}
        \begin{split}
            \|u - u_h\|_{2, h}
            &\lesssim \frac{c^*}{c_* - \mu} (\|u - u_I\|_2 + |u - \Pi_m^0 u|_{2, h}) \\
            &\quad + \frac{1}{c_* - \mu} \left(\sum_{K \in \cT_h} |K|\, |u - \Pi_m^0 u|_{2, \infty, E \cap K}^2\right)^{1/2}.
        \end{split}
    \end{equation}
\end{theorem}

\begin{proof}
    \Cref{lemma:quadr_partial_error_estimate} applies, so we only need to estimate the rightmost term in \cref{eq:quadr_partial_error_estimate}. By \cref{prop:consistency}, $A : \nabla^2 u = f$ almost everywhere in $\Omega$, and thus also in $E$. Let $K \in \cT_h$, $x \in X_K$, and let $(x_n)_{n \geq 0} \subset E \cap K$ be a sequence of points converging to $x$ and such that $A_K(x_n) : \nabla^2 u(x_n) = f_K(x_n)$ and $|\nabla^2 u(x_n) - \nabla^2 \Pi_m^0 u(x_n)| \leq |u - \Pi_m^0 u|_{2, \infty, E \cap K}$ for every $n \in \NN$. Then
    \begin{align*}
        |f_K(x_n) - A_K(x_n) : \nabla^2 \Pi_m^0 u(x_n)|
        &= |A_K(x_n) : (\nabla^2 u(x_n) - \nabla^2 \Pi_m^0 u(x_n))| \\
        &\leq |A_K(x_n)|\, |u - \Pi_m^0 u|_{2, \infty, E \cap K}.
    \end{align*}
    By the assumption \cref{eq:quadr_set_assumption}, $\lim_{n \to \infty} A_K(x_n) = f_K(x)$ and $\lim_{n \to \infty} f_K(x_n) = f_K(x)$, so
    \begin{align*}
        \gamma_K(x) |f_K(x) - A_K(x) : \nabla^2 \Pi_m^0 u(x)|
        &\leq \gamma_K(x) |A_K(x)|\, |u - \Pi_m^0 u|_{2, \infty, E \cap K} \\
        &\lesssim |u - \Pi_m^0 u|_{2, \infty, E \cap K},
    \end{align*}
    where we argued as in \cref{prop:cordes_scaled_bound} to control $|\gamma_K(x) A_K(x)|$. We deduce that, on any $K \in \cT_h$,
    \begin{align*}
        Q_K[|\gamma_K (f_K(x) - A_K(x) : \nabla^2 \Pi_m^0 u(x))|^2]
        &\leq |u - \Pi_m^0 u|_{2, \infty, E \cap K}^2\, Q_K[1] \\
        &= |K|\, |u - \Pi_m^0 u|_{2, \infty, E \cap K}^2,
    \end{align*}
    which concludes the proof.
\end{proof}

We refer to \cref{sec:h2vem_rates} for convergence rate results based on \cref{thm:quadr_error_estimate_discontinuous}, and \cref{subsec:numerics_roughdata_roughsol,subsec:numerics_roughdata_smoothsol} for examples of situations in which \cref{thm:quadr_error_estimate_discontinuous} applies but \cref{thm:quadr_error_estimate_continuous} does not.

\begin{remark}
    It may seem surprising that the regularity of $A$, $f$, and $\gamma$ is not involved in the error estimates \cref{eq:quadr_error_estimate_continuous,eq:quadr_error_estimate_discontinuous}. This is thanks to the fact that we were able to rely on the strong form of the equation $A : \nabla^2 u = f$ in the analysis, as in \cref{thm:error_estimate} (see \cref{rem:no_data_regularity}). Observe also that the analysis relies on the important assumption that, on each cell $K \in \cT_h$, the same quadrature rule $Q_K$ is used in the definitions of both the bilinear form $a_{h, K}^Q$ and the linear form $L_{h, K}^Q$.
\end{remark}

%%%%%%%%%%%%%%%%%%%%%%%%%%%%%%%%%%%%%%%%%%%%%%%%%%%%%%%%%%%%%%%%%%%%%%%%%%%%%%%%%%%%%%
\section{Realization of the virtual element framework}
\label{sec:h2vem}
%%%%%%%%%%%%%%%%%%%%%%%%%%%%%%%%%%%%%%%%%%%%%%%%%%%%%%%%%%%%%%%%%%%%%%%%%%%%%%%%%%%%%%

In this section, we detail a realization of the $H^2$-conforming virtual element discretization framework introduced in \cref{sec:discr}, and then discuss the properties of associated quasi-interpolation operators.

The development of $H^2$-conforming virtual element methods has been considered in a number of publications. Virtual elements for plate bending problems were proposed in Ref.~\refcite{brezzi2013}. Subsequently, virtual elements for forth order problems in two dimensions have been analyzed in Refs.~\refcite{antonietti2016,antonietti2021,antonietti2020,beirao2014} and extended to three dimensions in Refs.~\refcite{beirao2020,chen2022conforming}; see also Refs.~\refcite{brenner2019,chen2022hessian} for the special case of tetrahedral elements.

A complete family of virtual element spaces for all space dimensions was recently presented in Ref.~\refcite{chen2022conforming}. They use the so-called enhancement technique introduced in Ref.~\refcite{ahmad2013}, and their construction is hierarchical in the space dimension. We review this construction in \cref{subsec:h2vem_notation,subsec:h2vem_2d,subsec:h2vem_3d,subsec:h2vem_proj_stab}, and show, by referring to Ref.~\refcite{chen2022conforming} as appropriate, that it fits in the abstract discretization framework of \cref{sec:discr}. For simplicity, we detail here only the spaces obtained for $d\le 3$ and refer the reader to Ref.~\refcite{chen2022conforming} for the general case.

In \cref{subsec:h2vem_global_interp}, we present and analyze Scott-Zhang type and Lagrange type interpolation operators. A slightly different estimate can be found in Ref.~\cite{chen2022conforming}. However, this cannot be used directly in our context as it was designed for polyharmonic problems with homogeneous boundary conditions. Note also that our Lagrange type interpolation error estimate has a slightly unusual form, due to the fact that we want to be able to apply it in low regularity settings; we refer to \cref{subsubsec:h2vem_lagrange} for more details.

\subsection{Notation}
\label{subsec:h2vem_notation}

We require some extra notation. We shall use the symbols ${\rm v}$ and $e$ to indicate a generic vertex and edge of the partition, respectively, and the symbol $F$ for a face when $d=3$. Further, we assume that a complete set of unit normal vectors is given on every edge $e$ and face $F$, denoted by $\mathbf{n}_{e,i}$, $i=1,d-1$, and $\mathbf{n}_F$, respectively. Similarly, to every vertex $\rv$ we associate a basis of unit vectors $\bn_{\rv,i}$, $i=1,d$, of $\RR^d$. We do not assume orthogonality of the families $\bn_{\rv,i}$ and $\bn_{e,i}$; rather, we assume that for every vertex $\rv$ and $\bn \in \RR^d$, one has $|\bn|^2 \lesssim \sum_{i=1}^d \<\bn_{\rv,i}, \bn\>^2$, and for any edge $e$ and $\bn \in \spann \{\bn_{e,i} \mid i=1,d-1\}$, one has $|\bn|^2 \lesssim \sum_{i=1}^{d-1} \<\bn_{e,i}, \bn\>^2$. These nondegeneracy assumptions will be more convenient than orthogonality in \cref{subsec:h2vem_global_interp}.

For any $K \in \cT_h$, we denote by $\mathcal{V}_K$  the set of its vertices and by $\cE_K$ the set of its edges; further, when $d=3$, we denote by $\cF_K$ the set of faces of $K$. For a smooth enough function $v$ defined on $K$, being it scalar or vector valued, we set
 $\displaystyle\overline{v}:=\frac{1}{|\mathcal{V}_K|}\sum_{{\rm v}\in \mathcal{V}_K}v({\rm v})$.

We let $\cV_h := \bigcup_{K \in \cT_h} \cV_K$, $\cE_h := \bigcup_{K \in \cT_h} \cE_K$, and, if $d=3$, $\cF_h := \bigcup_{K \in \cT_h} \cF_K$. For ${\rm v}\in \mathcal{V}_h$, we let $h_{\rm v}$ be an appropriate local mesh size parameter associated to ${\rm v}$, such as the average diameter of the elements sharing ${\rm v}$ as a vertex.

Given $\omega\subset \RR^d$ and $k\in\NN$, we denote by $\PP_{k}(\omega)$ the space of polynomials of degree $k$ on $\omega$ and  let $\MM_{k}(\omega)$ denote a basis for such space. This notation is extended to $k<0$  by fixing $\PP_{k}(\omega)=\{0\}$ and  $\MM_{k}(\omega)=\{0\}$ in this case. The $L^2$-projection onto $\PP_{k}(\omega)$ will be denoted by $\Pi^{0,\omega}_k$.

The starting point for the sequential construction of virtual element spaces in two- and three-dimensions are one-dimensional spaces used on the edges composing the skeleton of the partition.
These are fixed as standard polynomial spaces as follows:  for a segment $e$ we consider the spaces  $V_{h,e}^0:=\PP_{m-1}(e)$ and $V_{h,e}:=\PP_{3\lor m}(e)$ typically used to construct the classical $C^j$-conforming, $j=0,1$,  finite elements in one dimension, respectively.

\subsection{Local virtual space in two-dimensions}
\label{subsec:h2vem_2d}

Let $d=2$ and $K\in \cT_h$. We first introduce the enlarged virtual element space
\begin{align*}
\widetilde{V}_{h,K}:=\left\{
v\in H^2(K)\, : \, \Delta^2 v\in \PP_m(K)
 \text{ and } v|_e\in V_{h,e},
\partial_{\mathbf{n}}v|_e\in V_{h,e}^0,
\forall e\in \cE_K
\right\}.
\end{align*}

\begin{remark}
    It may be surprising that we impose $\Delta^2 v \in \PP_m(K)$ rather than $\Delta^2 v \in \PP_{m-4}(K)$, but this is standard when using the virtual element enhancement technique. The reason is that imposing $\Delta^2 v \in \PP_{m-4}(K)$ would be too restrictive to be compatible with the additional constraints in equation \cref{eq:vem_local} below. We refer to Ref.~\refcite{ahmad2013} for more details.
\end{remark}

We clearly have $\PP_m(K)\subset \widetilde{V}_{h,K}$. Moreover, for any $v\in \widetilde{V}_{h,K}$, we have that $v$ and $\nabla v$ are continuous on $\partial K$ in consequence of the compatibility conditions implied by $v\in H^2(K)$, \emph{cf.}~Refs.~\refcite{antonietti2020,chen2022conforming,grisvard2011}.

We define the local Hessian projection operator $\Pi_m^{2}:H^2(K)\rightarrow\PP_m(K)$ by
\begin{align}\label{eq:hessproj1}
&\int_K \nabla^2 \Pi_m^{2} v : \nabla^2 q\, dx =\int_K \nabla^2 v : \nabla^2 q\, dx\quad &\forall q\in  \PP_m(K),\\
&
\overline{\nabla^j\Pi_m^{2}  v}=\overline{\nabla^j v}& j=0,1.
\label{eq:hessproj2}
\end{align}
If $v \in \widetilde{V}_{h,K}$, then the definition of $\widetilde{V}_{h,K}$
permits the evaluation of $\Pi_m^{2} v$ directly in function of the following (incomplete for $\widetilde{V}_{h,K}$) set of degrees of freedom.

\begin{definition}[two-dimensional degrees of freedom]\label{def:2Ddof}
On  $K\in \cT_h$, we define the \emph{local degrees of freedom}:
\begin{itemlist}
\item $h_{{\rm v}}^j\nabla^j v({\rm v})$ for all ${\rm v}\in\mathcal{V}_K$, $j=0,1$;
\item  $\displaystyle |e|^{j-1} \int_e \frac{\partial^j v}{\partial \mathbf{n}_{e,i}^j}
q \, d\ell$, for all  $q\in\MM_{m+j -4}(e)$, $j=0,1$, $i=1,d-1$, and $e\in\cE_K$;
\item  $\displaystyle |K|^{-1}\int_K v q\, dx$, for all $q\in\MM_{m-4}(K)$.
\end{itemlist}
\end{definition}

A depiction of the  degrees of freedom corresponding to $m=2,3,4$ is shown in \cref{C1-dof-2D}.

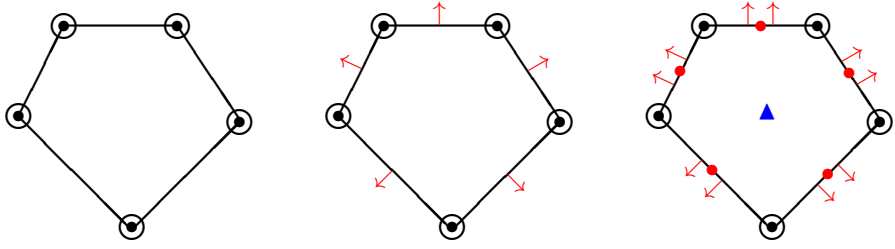
\begin{figure}[ht]
\centering
\setlength{\unitlength}{1.5cm}
\begin{picture}(2,2)
	\thicklines
	\put(0,1){\line(1,2){.4}}
	\put(.4,1.8){\line(1,0){1}}
	\put(1.4,1.8){\line(2,-3){.55}}
	\put(0,1){\line(1,-1){1}}
	\put(1,0){\line(1,1){.96}}
	\put(0,1){\circle*{.1}}
	\put(.4,1.8){\circle*{.1}}
	\put(1.4,1.8){\circle*{.1}}
	\put(1.95,.95){\circle*{.1}}
	\put(1,0.03){\circle*{.1}}
	\put(0,1){{\color{black} \circle{.2}}}
	\put(.4,1.8){{\color{black} \circle{.2}}}
	\put(1.4,1.8){{\color{black} \circle{.2}}}
	\put(1.95,.95){{\color{black} \circle{.2}}}
	\put(1,0.03){{\color{black} \circle{.2}}}
\end{picture}
\hspace{1cm}
\begin{picture}(2,2)
	\thicklines
	\put(0,1){\line(1,2){.4}}
	\put(.4,1.8){\line(1,0){1}}
	\put(1.4,1.8){\line(2,-3){.55}}
	\put(0,1){\line(1,-1){1}}
	\put(1,0){\line(1,1){.96}}
	\put(0,1){\circle*{.1}}
	\put(.4,1.8){\circle*{.1}}
	\put(1.4,1.8){\circle*{.1}}
	\put(1.95,.95){\circle*{.1}}
	\put(1,0.03){\circle*{.1}}
	\put(0,1){{\color{black} \circle{.2}}}
	\put(.4,1.8){{\color{black} \circle{.2}}}
	\put(1.4,1.8){{\color{black} \circle{.2}}}
	\put(1.95,.95){{\color{black} \circle{.2}}}
	\put(1,0.03){{\color{black} \circle{.2}}}
	\put(-.02,1.41){{\color{red}  \rotatebox[origin=c]{-25}{$\leftarrow$}}}
	\put(.83,1.85){{\color{red}  {$\uparrow$}}}
	\put(1.65,1.4){{\color{red}  \rotatebox[origin=c]{30}{$\rightarrow$}}}
	\put(1.44,.36){{\color{red}  \rotatebox[origin=c]{-45}{$\rightarrow$}}}
	\put(.3,.4){{\color{red}  \rotatebox[origin=c]{45}{$\leftarrow$}}}
\end{picture}
\hspace{1cm}
\begin{picture}(2,2)
	\thicklines
	\put(0,1){\line(1,2){.4}}
	\put(.4,1.8){\line(1,0){1}}
	\put(1.4,1.8){\line(2,-3){.55}}
	\put(0,1){\line(1,-1){1}}
	\put(1,0){\line(1,1){.96}}
	\put(0,1){\circle*{.1}}
	\put(.4,1.8){\circle*{.1}}
	\put(1.4,1.8){\circle*{.1}}
	\put(1.95,.95){\circle*{.1}}
	\put(1,0.03){\circle*{.1}}
	\put(0,1){{\color{black} \circle{.2}}}
	\put(.4,1.8){{\color{black} \circle{.2}}}
	\put(1.4,1.8){{\color{black} \circle{.2}}}
	\put(1.95,.95){{\color{black} \circle{.2}}}
	\put(1,0.03){{\color{black} \circle{.2}}}
	\put(.02,1.49){{\color{red}  \rotatebox[origin=c]{-25}{$\leftarrow$}}}
	\put(-.09,1.27){{\color{red}  \rotatebox[origin=c]{-25}{$\leftarrow$}}}
	\put(.73,1.85){{\color{red}  {$\uparrow$}}}
	\put(.95,1.85){{\color{red}  {$\uparrow$}}}
	\put(1.59,1.48){{\color{red}  \rotatebox[origin=c]{30}{$\rightarrow$}}}
	\put(1.72,1.28){{\color{red}  \rotatebox[origin=c]{30}{$\rightarrow$}}}
	\put(1.53,.45){{\color{red}  \rotatebox[origin=c]{-45}{$\rightarrow$}}}
	\put(1.35,.27){{\color{red}  \rotatebox[origin=c]{-45}{$\rightarrow$}}}
	\put(.38,.31){{\color{red}  \rotatebox[origin=c]{45}{$\leftarrow$}}}
	\put(.21,.49){{\color{red}  \rotatebox[origin=c]{45}{$\leftarrow$}}}

	\put(.19,1.4){{\color{red}\circle*{.1}}}
	\put(.9,1.8){{\color{red}\circle*{.1}}}
	\put(1.68,1.38){{\color{red}\circle*{.1}}}
	\put(1.49,.49){{\color{red}\circle*{.1}}}
	\put(.47,.53){{\color{red}\circle*{.1}}}
	\put(.87,.98){{\color{blue}$\blacktriangle$}}
\end{picture}
\caption{Degrees of freedom of the $H^2$-conforming VEM for $m=2,3,4$.}
\label{C1-dof-2D}
\end{figure}

\begin{proposition}[computability]\label{prop:hessian}
For any $v\in\widetilde{V}_{h,K}$, the  Hessian projection $\Pi_m^{2} v$ and $L^2$-projection of the Hessian $\Pi_{m-2}^0 \nabla^2 v$ can be evaluated from the set of degrees of freedom of \cref{def:2Ddof}.
\end{proposition}
\begin{proof}
Noting the formula
\[
\int_K \nabla^2 v : \nabla^2 q\, dx =\int_K v  \Delta^2 q\, dx-\sum_{e\in\cE_K}\int_{e} v (\nabla\cdot\nabla^2q)\cdot\mathbf{n}_e\, d\ell
+\sum_{e\in\cE_K}\int_{e} \nabla v \cdot (\nabla^2q\mathbf{n}_e)\, d\ell,
\]
we have that the evaluation of the right-hand side of~\cref{eq:hessproj1} only requires the knowledge of the following moments: the internal moments of $v$ of order up to $m-4$, the edge  moments of $v$ of order up to $m-3$, and the edge  moments of $\partial v/\partial\mathbf{n}_e$ of order up to $m-2$.
On the other hand,  the evaluation  the right-hand side of~\cref{eq:hessproj2} only requires the knowledge of $v$ and its gradient at the vertices.
The proof that   $\Pi_{m-2}^0 \nabla^2 v$ is also computable from the given degrees of freedom of $v$ is similar; see e.g. Ref.~\refcite{cangiani2017} for an analogue result.
\end{proof}

The local $H^2$-conforming virtual element space may now be defined as
\begin{align}\label{eq:vem_local}
V_{h,K}:=\left\{
v\in \widetilde{V}_{h,K}\, : \int_K (v- \Pi_m^{2}v)q\, dx=0, \forall q\in \PP_{m}(K)\setminus\PP_{m-4}(K)
\right\}.
\end{align}
Clearly, we still have  $\PP_{m}(K)\subset V_{h,K}$,
and it is easy to see that the dimension of the space $V_{h,K}$ equals the number of degrees of freedom listed above and hence these may be used as degrees of freedom for $V_{h,K}$; see Ref.~\refcite{chen2022conforming} for the unisolvency proof.
Crucially, given that $\Pi_m^{2}v$ only depends on such degrees of freedom, $V_{h,K}$ is a proper linear subspace of $V_{h,K}$.
Moreover, from the equality  $ \Pi_m^0 v = ( I_d- \Pi_{m-4}^0 ) \Pi_m^{2} v+\Pi_{m-4}^0 v$, we  deduce the computability of $ \Pi_m^0$ on $V_{h,K}$ as well.

\begin{remark}
The enhancement technique leading to the definition of $V_{h,K}$ is required to achieve the computability of $\Pi_m^0 v$. As such, we note that the enhancement technique is not strictly necessary for the discretization of our model problem in two-dimensions as the computability of  $\Pi_{m-2}^0 \nabla^2 v$ suffices in this case, and this is guaranteed also without enhancement. However, we do require the enhanced space in view of defining  the face values of the three-dimensional discrete functions. Moreover, the enhancement  is required if lower-order terms are present in the model, \emph{cf.}~Ref.~\refcite{bcdn}.
\end{remark}

\subsection{Local virtual space in three-dimensions}
\label{subsec:h2vem_3d}

Let $d=3$ and $K\in \cT_h$. We start by fixing the traces of virtual functions over the boundary of $K$ using  two-dimensional virtual element spaces. In particular, the space $V_{h,F}$, $F\in\cF_K$,  defined according to~\cref{eq:vem_local}  is used for the value trace and the standard $H^1$-conforming two-dimensional virtual element space $V_{h,F}^0$ is used for the normal derivative. We introduce the latter here briefly, referring to e.g. Refs.~\refcite{ahmad2013,cangiani2017,chen2022conforming} for the details.

Let $F\in\cF_K$. We first define an enlarged virtual element space as
\begin{align*}
\widetilde{V}_{h,F}^0:=\left\{
v\in H^1(F)\, : \, \Delta v\in \PP_{m-1}(F)  \text{ and } v|_e\in\PP_{m-1}(e), \forall e\in \Gamma_F
\right\}.
\end{align*}
We consider  the local gradient projection $ \Pi_{m-1}^1:H^1(F)\rightarrow \PP_{m-1}(F) $ given by
$\int_F \nabla \Pi_{m-1}^1 v \cdot \nabla q\, dx =\int_F \nabla v \cdot \nabla q\, dx$ for all $q\in  \PP_{m-1}(F)$ and
$\overline{\Pi_{m-1}^1  v}=\overline{v}$. When $v \in \widetilde{V}_{h,F}^0$, the projection $\Pi_{m-1}^1 v$ is proven to be  computable through the local degrees of freedom: (i) the value $v({\rm v})$ for all ${\rm v}\in\mathcal{V}_F$, (ii) $|e|^{-1} \int_e v q \, ds$ for all  $q\in\MM_{m-3}(e)$ and $e\in\Gamma_F$, and
(iii)  $\frac{1}{|F|}\int_F v q\, dx$, for all $q\in\MM_{m-3}(F)$. Then, the local $H^1$-conforming space is defined as
\begin{align*}
V_{h,F}^0:=\{
v\in \widetilde{V}_{h,F}^0\, : \int_F (v- \Pi_{m-1}^1 v)q\, dx=0, \forall q\in \PP_{m-1}(F)\setminus\PP_{m-3}(F)
\},
\end{align*}
having the above as degrees of freedom. Such space is characterized by the fact that $ \PP_{m-1}(F)\subset V_{h,F}^0$ and the property of computability, beside  $ \Pi_{m-1}^1$, of the $L^2$-projection $\Pi_{m-1}^0$.

Having the virtual trace spaces sorted, we can now implement once more the enhancement technique to define an $H^2$-conforming virtual element over $K\in \cT_h$. First, the enlarged virtual element space is defined by
\begin{align*}
\widetilde{V}_{h,K}:=\left\{\right.
v\in H^2(K)\, : &\,  \nabla^j v|_{\Gamma_K}\in H^1(\Gamma_K), \nabla^j
v|_{\partial K}\in H^1(\partial K), j=0,1, \text{and}
\\
&\left. \Delta^2 v\in \PP_m(K),
v|_F\in V_{h,F},
\partial_{\mathbf{n}_{F}}v|_F\in V_{h,F}^0,
\forall F\in \cF_K
\right\},
\end{align*}
where $\Gamma_K := \bigcup_{e \in \cE_K} \overline e$; observe that similarly, $\partial K = \bigcup_{F \in \cF_K} \overline F$. Once again we have $\PP_{m}\subset \widetilde{V}_{h,K}$ and continuity  on $\partial K$ up to the gradient of all functions in  $\widetilde{V}_{h,K}$, \emph{cf.}~Ref.~\refcite{chen2022conforming}.
It is easy to prove once again \cref{prop:hessian}, namely  the  computability of the Hessian projector~\cref{eq:hessproj1},~\cref{eq:hessproj2}, given the following (incomplete for $\widetilde{V}_{h,K}$) degrees of freedom.
\begin{definition}[three-dimensional degrees of freedom]\label{def:3Ddof}
On  $K\in \cT_h$, we consider the \emph{local degrees of freedom} as those listed in \cref{def:2Ddof} plus
\begin{itemlist}
\item $\displaystyle |F|^{j/2-1} \int_F \frac{\partial^j v}{\partial \mathbf{n}_F^j} q \, ds$, for all  $q\in\MM_{m+j-4}(F)$, $j=0,1$ and $F\in\cF_K$.
\end{itemlist}
\end{definition}
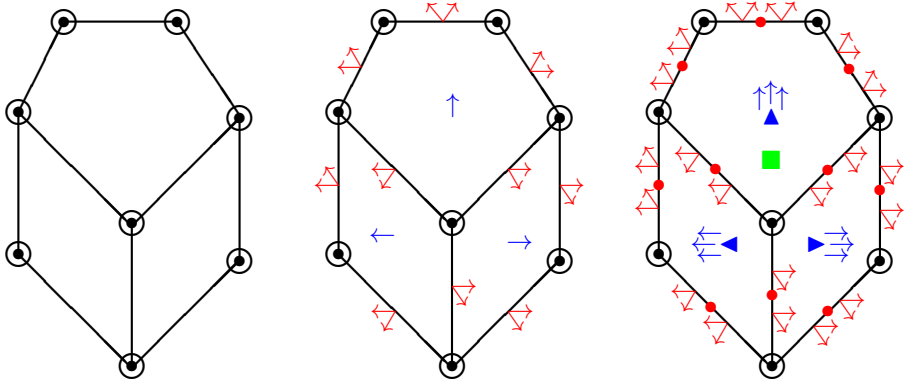
\begin{figure}[ht]
\centering
\setlength{\unitlength}{1.5cm}
\begin{picture}(2,3)
	\thicklines
	\put(0,2){\line(1,2){.4}}
	\put(.4,2.8){\line(1,0){1}}
	\put(1.4,2.8){\line(2,-3){.55}}
	\put(0,2){\line(1,-1){1}}
	\put(1,1){\line(1,1){.96}}
	\put(0,2){\circle*{.1}}
	\put(.4,2.8){\circle*{.1}}
	\put(1.4,2.8){\circle*{.1}}
	\put(1.95,1.95){\circle*{.1}}
	\put(1,1.03){\circle*{.1}}
	\put(0,2){{\color{black} \circle{.2}}}
	\put(.4,2.8){{\color{black} \circle{.2}}}
	\put(1.4,2.8){{\color{black} \circle{.2}}}
	\put(1.95,1.95){{\color{black} \circle{.2}}}
	\put(1,1.03){{\color{black} \circle{.2}}}
	\put(0,.75){\line(0,1){1.3}}
	\put(1,-.25){\line(0,1){1.3}}
	\put(1.95,.68){\line(0,1){1.28}}
	\put(0,.75){\line(1,-1){1}}
	\put(1,-.25){\line(1,1){.96}}
	\put(0,.75){\circle*{.1}}
	\put(1,-.23){\circle*{.1}}
	\put(1.95,.69){\circle*{.1}}
	\put(0,.75){\circle{.2}}
	\put(1,-.23){\circle{.2}}
	\put(1.95,.69){\circle{.2}}
\end{picture}
\hspace{1cm}
\begin{picture}(2,3)
	\thicklines
	\put(0,2){\line(1,2){.4}}
	\put(.4,2.8){\line(1,0){1}}
	\put(1.4,2.8){\line(2,-3){.55}}
	\put(0,2){\line(1,-1){1}}
	\put(1,1){\line(1,1){.96}}
	\put(0,2){\circle*{.1}}
	\put(.4,2.8){\circle*{.1}}
	\put(1.4,2.8){\circle*{.1}}
	\put(1.95,1.95){\circle*{.1}}
	\put(1,1.03){\circle*{.1}}
	\put(0,2){{\color{black} \circle{.2}}}
	\put(.4,2.8){{\color{black} \circle{.2}}}
	\put(1.4,2.8){{\color{black} \circle{.2}}}
	\put(1.95,1.95){{\color{black} \circle{.2}}}
	\put(1,1.03){{\color{black} \circle{.2}}}
	\put(0,.75){\line(0,1){1.3}}
	\put(1,-.25){\line(0,1){1.3}}
	\put(1.95,.68){\line(0,1){1.28}}
	\put(0,.75){\line(1,-1){1}}
	\put(1,-.25){\line(1,1){.96}}
	\put(0,.75){\circle*{.1}}
	\put(1,-.23){\circle*{.1}}
	\put(1.95,.69){\circle*{.1}}
	\put(0,.75){\circle{.2}}
	\put(1,-.23){\circle{.2}}
	\put(1.95,.69){\circle{.2}}
	\put(-.02,2.34){{\color{red}  {$\leftarrow$}}}
	\put(.04,2.44){{\color{red}  \rotatebox[origin=c]{-60}{$\leftarrow$}}}
	\put(.73,2.83){{\color{red}   \rotatebox[origin=c]{-50}{$\leftarrow$}}}
	\put(.89,2.83){{\color{red}  \rotatebox[origin=c]{55}{$\rightarrow$}}}
	\put(1.66,2.40){{\color{red}  \rotatebox[origin=c]{60}{$\rightarrow$}}}
	\put(1.68,2.3){{\color{red}  \rotatebox[origin=c]{0}{$\rightarrow$}}}
	%%% 1
	\put(1.48,1.44){{\color{red}  \rotatebox[origin=c]{0}{$\rightarrow$}}}
	\put(1.44,1.36){{\color{red}  \rotatebox[origin=c]{-60}{$\rightarrow$}}}
	% 2
	\put(.35,1.36){{\color{red}  \rotatebox[origin=c]{55}{$\leftarrow$}}}
	\put(.28,1.44){{\color{red}  \rotatebox[origin=c]{0}{$\leftarrow$}}}
	\put(-.17,1.39){{\color{red}  \rotatebox[origin=c]{-55}{$\leftarrow$}}}
	\put(-.22,1.3){{\color{red}  {$\leftarrow$}}}
	\put(1.9,1.22){{\color{red}  \rotatebox[origin=c]{-55}{$\rightarrow$}}}
	\put(1.94,1.3){{\color{red}  {$\rightarrow$}}}
	\put(.95,.33){{\color{red}  \rotatebox[origin=c]{-55}{$\rightarrow$}}}
	\put(.99,.41){{\color{red}  {$\rightarrow$}}}
	\put(1.44,.11){{\color{red}  \rotatebox[origin=c]{-55}{$\rightarrow$}}}
	\put(1.49,.19){{\color{red}  {$\rightarrow$}}}
	\put(.37,.11){{\color{red}  \rotatebox[origin=c]{60}{$\leftarrow$}}}
	\put(.28,.19){{\color{red}  {$\leftarrow$}}}
	% face
	\put(.95,2){{\color{blue} $\uparrow$}}
	\put(1.48,.8){{\color{blue}$\rightarrow$}}
	\put(.27,.86){{\color{blue}$\leftarrow$}}
\end{picture}
\hspace{1cm}
\begin{picture}(2,3)
	\thicklines
	\put(0,2){\line(1,2){.4}}
	\put(.4,2.8){\line(1,0){1}}
	\put(1.4,2.8){\line(2,-3){.55}}
	\put(0,2){\line(1,-1){1}}
	\put(1,1){\line(1,1){.96}}
	\put(0,2){\circle*{.1}}
	\put(.4,2.8){\circle*{.1}}
	\put(1.4,2.8){\circle*{.1}}
	\put(1.95,1.95){\circle*{.1}}
	\put(1,1.03){\circle*{.1}}
	\put(0,2){{\color{black} \circle{.2}}}
	\put(.4,2.8){{\color{black} \circle{.2}}}
	\put(1.4,2.8){{\color{black} \circle{.2}}}
	\put(1.95,1.95){{\color{black} \circle{.2}}}
	\put(1,1.03){{\color{black} \circle{.2}}}
	\put(0,.75){\line(0,1){1.3}}
	\put(1,-.25){\line(0,1){1.3}}
	\put(1.95,.68){\line(0,1){1.28}}
	\put(0,.75){\line(1,-1){1}}
	\put(1,-.25){\line(1,1){.96}}
	\put(0,.75){\circle*{.1}}
	\put(1,-.23){\circle*{.1}}
	\put(1.95,.69){\circle*{.1}}
	\put(0,.75){\circle{.2}}
	\put(1,-.23){\circle{.2}}
	\put(1.95,.69){\circle{.2}}
	\put(.21,2.41){{\color{red}\circle*{.1}}}
	\put(.9,2.8){{\color{red}\circle*{.1}}}
	\put(1.68,2.38){{\color{red}\circle*{.1}}}
	\put(1.49,1.49){{\color{red}\circle*{.1}}}
	\put(.5,1.5){{\color{red}\circle*{.1}}}
	\put(0,1.36){{\color{red}\circle*{.1}}}
	\put(1,.39){{\color{red}\circle*{.1}}}
	\put(1.95,1.32){{\color{red}\circle*{.1}}}
	\put(.46,.29){{\color{red}\circle*{.1}}}
	\put(1.49,.25){{\color{red}\circle*{.1}}}
	% 0
	\put(.06,2.51){{\color{red}  {$\leftarrow$}}}
	\put(.12,2.61){{\color{red}  \rotatebox[origin=c]{-60}{$\leftarrow$}}}
	\put(-.09,2.2){{\color{red}  {$\leftarrow$}}}
	\put(-.03,2.3){{\color{red}  \rotatebox[origin=c]{-60}{$\leftarrow$}}}
	% 1
	\put(.90,2.83){{\color{red}  \rotatebox[origin=c]{-50}{$\leftarrow$}}}
	\put(1.05,2.84){{\color{red}  \rotatebox[origin=c]{55}{$\rightarrow$}}}
	\put(.55,2.83){{\color{red}  \rotatebox[origin=c]{-50}{$\leftarrow$}}}
	\put(.71,2.84){{\color{red}  \rotatebox[origin=c]{55}{$\rightarrow$}}}
	% 2
	\put(1.79,2.14){{\color{red}   {$\rightarrow$}}}
	\put(1.77,2.23){{\color{red}  \rotatebox[origin=c]{60}{$\rightarrow$}}}
	\put(1.55,2.57){{\color{red}  \rotatebox[origin=c]{60}{$\rightarrow$}}}
	\put(1.56,2.47){{\color{red} {$\rightarrow$}}}
	% 3
	\put(1.58,1.5){{\color{red}  \rotatebox[origin=c]{-55}{$\rightarrow$}}}
	\put(1.62,1.58){{\color{red}  {$\rightarrow$}}}
	\put(1.28,1.2){{\color{red}  \rotatebox[origin=c]{-55}{$\rightarrow$}}}
	\put(1.32,1.27){{\color{red}  {$\rightarrow$}}}
	% 4
	\put(.49,1.22){{\color{red}  \rotatebox[origin=c]{55}{$\leftarrow$}}}
	\put(.42,1.3){{\color{red}  {$\leftarrow$}}}
	\put(.21,1.5){{\color{red}  \rotatebox[origin=c]{55}{$\leftarrow$}}}
	\put(.13,1.58){{\color{red}  {$\leftarrow$}}}
	\put(-.18,1.56){{\color{red}  \rotatebox[origin=c]{-55}{$\leftarrow$}}}
	\put(-.22,1.47){{\color{red}  {$\leftarrow$}}}
	\put(-.18,1.19){{\color{red}  \rotatebox[origin=c]{-55}{$\leftarrow$}}}
	\put(-.22,1.1){{\color{red}  {$\leftarrow$}}}
	\put(1.9,1.4){{\color{red}  \rotatebox[origin=c]{-55}{$\rightarrow$}}}
	\put(1.94,1.47){{\color{red}  {$\rightarrow$}}}
	\put(1.9,1){{\color{red}  \rotatebox[origin=c]{-55}{$\rightarrow$}}}
	\put(1.94,1.08){{\color{red}  {$\rightarrow$}}}
	\put(.95,.46){{\color{red}  \rotatebox[origin=c]{-55}{$\rightarrow$}}}
	\put(.99,.54){{\color{red}  {$\rightarrow$}}}
	\put(.95,.09){{\color{red}  \rotatebox[origin=c]{-55}{$\rightarrow$}}}
	\put(.99,.17){{\color{red}  {$\rightarrow$}}}
	\put(1.54,.21){{\color{red}  \rotatebox[origin=c]{-55}{$\rightarrow$}}}
	\put(1.59,.29){{\color{red}  {$\rightarrow$}}}
	\put(1.31,-.02){{\color{red}  \rotatebox[origin=c]{-55}{$\rightarrow$}}}
	\put(1.36,.06){{\color{red}  {$\rightarrow$}}}
	\put(.19,.29){{\color{red}  \rotatebox[origin=c]{60}{$\leftarrow$}}}
	\put(.1,.37){{\color{red}  {$\leftarrow$}}}
	\put(.47,.0){{\color{red}  \rotatebox[origin=c]{60}{$\leftarrow$}}}
	\put(.38,.09){{\color{red}  {$\leftarrow$}}}
	% face
	\put(.91,1.9){{\color{blue}$\blacktriangle$}}
	\put(.93,2.12){{\color{blue} $\uparrow$}}
	\put(1.03,2.07){{\color{blue} $\uparrow$}}
	\put(.83,2.07){{\color{blue} $\uparrow$}}

	\put(1.3,.78){{\color{blue}$\blacktriangleright$}}
	\put(1.5,.78){{\color{blue}$\rightarrow$}}
	\put(1.45,.68){{\color{blue}$\rightarrow$}}
	\put(1.45,.88){{\color{blue}$\rightarrow$}}

	\put(.53,.78){{\color{blue}$\blacktriangleleft$}}
	\put(.28,.78){{\color{blue}$\leftarrow$}}
	\put(.33,.68){{\color{blue}$\leftarrow$}}
	\put(.33,.88){{\color{blue}$\leftarrow$}}

	% element
	\put(.9,1.5){{\color{green}$\blacksquare$}}
\end{picture}

\begin{picture}(0,.1)
\end{picture}
	\caption{Degrees of freedom of the $H^2$-conforming VEM for $m=2,3,4$.}
	\label{C1-dof-3D}
\end{figure}

The degrees of freedom corresponding to $m=2,3,4$ are shown in \cref{C1-dof-3D}.
The local virtual element space $V_{h,K}$ for $K\in\cT_h$ when $d=3$ can now be defined as in~\cref{eq:vem_local}. The proof that the set of degrees of freedom in \cref{def:3Ddof} is unisolvent in $V_{h,K}$ and that they allow the evaluation of $\Pi_m^0$ alongside $\Pi_m^{2}$ can be found in Ref.~\refcite{chen2022conforming}.

\subsection{Local projection operator and stabilization term}
\label{subsec:h2vem_proj_stab}

In view of completing the construction of a VEM according to \cref{sec:discr}, it remains to fix a local projection operator $\Pi_m^*$ satisfying~\cref{eq:proj_stability} and a local stabilization form $s_{h, K}$ satisfying~\cref{eq:stab_scaling}.

The local Hessian projection is used for the projection operator $\Pi_m^*$ of \cref{sec:discr}; hence, we fix
\[
\Pi_m^*:= \Pi_m^{2}.
\]
For any $u \in V_{h, K}$, one has $\|\Pi_m^2 u\|_{2, K} = \|\Pi_m^0 u + \Pi_m^2 (u - \Pi_m^0 u)\|_{2, K} \lesssim \|u\|_{2, K} + \|\Pi_m^2 (u - \Pi_m^0 u)\|_{2, K}$. It is easy to deduce from (4.11) in Ref.~\refcite{chen2022conforming}, the trace inequality, and classical polynomial projection estimates\cite{brenner2008}, that $\|\Pi_m^2 (u - \Pi_m^0 u)\|_{2, K} \lesssim |u|_{2, K}$, from which we deduce that \cref{eq:proj_stability} holds.

% $\|\Pi_m^2 v\|_{2, K} \lesssim h^{-2} \|\Pi_{m-4}^0 v\|_{2, K}$

For any $u$, $v \in V_{h, K}$,  we let
\begin{equation*}
	\begin{split}
		s_{h, K}^{0}(u,v)&:=\sum_{{\rm v}\in\mathcal{V}_K}\sum_{j=0}^1h_{K}^{2j}\,\left(\nabla^ju({\rm v})\right)
		\cdot \left(\nabla^jv({\rm v})\right),
		\\
		s_{h, K}^{1}(u,v)&:=\sum_{e\in\cE_K}\sum_{j=0}^1
		h_{K}^{2j-1}
		\sum_{i=1}^{d-1}
		\int_e
		\Big(\Pi^{0,e}_{m+j-4}\frac{\partial^j u}{\partial \mathbf{n}_{e,i}^j}\Big)
		\Big(
		\Pi^{0,e}_{m+j-4}\frac{\partial^j v}{\partial \mathbf{n}_{e,i}^j} \Big)
		\, d\ell,
	\end{split}
\end{equation*}
and, when $d=3$, we also introduce
\[
s_{h, K}^{2}(u,v):=
		\sum_{F\in\cF_K}
		\sum_{j=0}^1
		h_{K}^{2j-2}
		\int_F
		\Big(\Pi^{0,F}_{m+j-4}\frac{\partial^j u}{\partial \mathbf{n}_{F}^j}\Big)
		\Big(
		\Pi^{0,F}_{m+j-4}\frac{\partial^j v}{\partial \mathbf{n}_{F}^j} \Big)
		\, ds.
\]
Note in particular that only the form $s_{h, K}^{0}$ is non-zero when $m=2$. We thus fix the local stabilization form as
\[
s_{h, K}(u,v):=\beta h_K^{d-4}
\sum_{l=0}^{d-1}s_{h, K}^{l}(u,v),
\]
where $\beta > 0$ is some arbitrary scaling factor. It is shown in Ref.~\refcite{chen2022conforming} (see equation (5.2) therein) that, for all $v\in V_{h,K}$,
\begin{equation}
    \label{eq:stab_scaling_from_ref}
    s_{h, K}(v-\Pi_m^{2} v,v-\Pi_m^{2} v)\eqsim \beta \|\nabla^2 (v-\Pi_m^{2} v)\|_{0, K}^2.
\end{equation}
Both inequalities in \cref{eq:stab_scaling} immediately follow, with constants $c_*$ and $c^*$ proportional to $\beta$. This implies that this choice of the stabilization form is admissible.

The inequality $c_* > \mu$, required in our proof of well-posedness of the scheme, holds provided that $\beta$ is large enough. Computing some value of $\beta$ for which the inequality $c_* > \mu$ is theoretically guaranteed would require a detailed study of the constants in the proof of \cref{eq:stab_scaling_from_ref}. Obtaining a sharp estimate of the critical value of $\beta$ is likely difficult. In practice, we simply fixed $\beta = 1$ in our numerical experiments, which led to satisfactory results (see \cref{sec:numerics}).

\subsection{Global spaces and quasi-interpolation operators}
\label{subsec:h2vem_global_interp}

It is clear that the local virtual element spaces defined above allow for the construction of a global space $V_h$ as a subspace of $H^2(\Omega)$ according to~\cref{eq:vemspace}. The elements of $V_h$ may be identified in standard fashion by the collection of the local degrees of freedom in \cref{def:2Ddof,def:3Ddof}. Then, the space $V_h^0$ may be defined according to \cref{eq:vemspacedirichlet}, which amounts to fixing the relevant boundary degrees of freedom. Likewise, the space $V_h^{g_I}$ may be defined according to \cref{eq:vemspacedirichlet} once the function $g_I \in V_h$ is chosen.

Let us now describe interpolation operators suitable for use with the virtual element spaces $V_h$, $V_h^0$, and $V_h^{g_I}$.

\subsubsection{Scott-Zhang type interpolation}
\label{subsubsec:h2vem_scott_zhang}

We aim to design an interpolation operator $I_h^{\SZ} \colon H^2(\Omega) \to V_h$ satisfying the following property:
\begin{equation}
    \label{eq:interp_boundary_dependence}
    \begin{split}
        &\text{\textit{for any $v$, $w \in H^2(\Omega)$ such that $v|_{\partial \Omega} = w|_{\partial \Omega}$,}} \\
        &\text{\textit{one has $(I_h^{\SZ} v)|_{\partial \Omega} = (I_h^{\SZ} w)|_{\partial \Omega}$.}}
    \end{split}
\end{equation}
This operator can be used in order to choose $g_I$ as $g_I := I_h^{\SZ} g$, and $u_I$ as $u_I := I_h^{\SZ} u$ in the error estimates \cref{eq:error_estimate,eq:quadr_error_estimate_continuous,eq:quadr_error_estimate_discontinuous}. The property \cref{eq:interp_boundary_dependence} guarantees that $u_I$ coincides with $g_I$ on $\partial \Omega$ and thus belongs to $V_h^{g_I}$. %Observe that, while here we choose for full generality to describe a Scott-Zhang type interpolation operator, for problems whose data and solution have more regularity one could use instead a simpler, Lagrange type interpolation operator; this is in fact how $g_I$ is chosen in our numerical experiments in \cref{sec:numerics}.

Following the Scott-Zhang construction, we consider the set of \emph{sides} $\Sigma_K$ associated to any $K \in \cT_h$, \emph{i.e.}~$\Sigma_K := \cE_K$ if $d=2$ and $\Sigma_K := \cF_K$ if $d=3$. We let $\Sigma_h := \bigcup_{K \in \cT_h} \Sigma_K$ and denote by $\bn_\sigma$ a unit normal vector to each $\sigma \in \Sigma_h$.

To every vertex $\rv \in \cV_h$, we associate sides $\sigma_{\rv,i} \in \Sigma_h$, $i=0,d$, such that $\rv \in \partial \sigma_{\rv,i}$ (conceptually, $\sigma_{\rv,0}$ will be used to fix the value degree of freedom $v(\rv)$ and $\sigma_{\rv,1}$, $\ldots$, $\sigma_{\rv,d}$ will be used to fix the gradient degree of freedom $\nabla v(\rv)$; why these sides may need to be chosen differently will become apparent below when we explain how to enforce the property \cref{eq:interp_boundary_dependence}). Similarly, when $d=3$, to every edge $e \in \cE_h$ we associate sides $\sigma_{e,i} \in \Sigma_h$, $i=0,d-1$, such that $e \subset \partial \sigma_{e,i}$.

In dimension $d=2$, for any $v \in H^2(\Omega)$ and $K \in \cT_h$, we define the local interpolant $I_K^{\SZ} v \in V_{h, K}$ as the function in $V_{h, K}$ satisfying:
\begin{itemlist}
    \item $\Pi_{m-4}^{0,K} I_K^{\SZ} v = \Pi_{m-4}^{0,K} v$;
    \item For any $\sigma \in \Sigma_K$, $\Pi_{m-4}^{0,\sigma} I_K^{\SZ} v = \Pi_{m-4}^{0,\sigma} v$ and $\Pi_{m-3}^{0,\sigma} \partial_{\bn_\sigma} I_K^{\SZ} v = \Pi_{m-3}^{0,\sigma} \partial_{\bn_\sigma} v$;
    \item For any $\rv \in \cV_K$, $(I_K^{\SZ} v)(\rv) = (\Pi_m^{0, \sigma_{\rv,0}} v)(\rv)$ and $(\partial_{\bn_{\rv,i}} I_K^{\SZ} v)(\rv) = (\Pi_{m-1}^{0, \sigma_{\rv,i}} \partial_{\bn_{\rv,i}} v)(\rv)$ for $i=1,d$.
\end{itemlist}
In dimension $d=3$, we keep the three above conditions and add the following:
\begin{itemlist}
    \item For any $e \in \cE_K$: $\Pi_{m-4}^{0,e} I_K^{\SZ} v = \Pi_{m-4}^{0,e} \Pi_m^{0, \sigma_{e,0}} v$ and $\Pi_{m-3}^{0,e} \partial_{\bn_{e,i}} I_K^{\SZ} v = \Pi_{m-3}^{0,e} \Pi_{m-1}^{0, \sigma_{e,i}} \partial_{\bn_{e,i}} v$ for $i=1,d-1$.
\end{itemlist}
It is easily verified that this uniquely determines each one of the degrees of freedom characterizing $I_K^{\SZ} v$. We then define $I_h^{\SZ} v \in V_h$ by $(I_h^{\SZ} v)|_K = I_K^{\SZ} v$ in every $K \in \cT_h$.

To establish \cref{eq:interp_boundary_dependence}, we need to make some additional assumptions. First we require that for any boundary vertex $\rv \in \cV_h \cap \partial \Omega$, the side $\sigma_{\rv,0}$ is included in $\partial \Omega$, and, if $d=3$ that for any boundary edge $e \in \cE_h$, $e \subset \partial \Omega$, the side $\sigma_{e,0}$ is included in $\partial \Omega$. This is not enough: for instance, if $\rv \in \cV_h \cap \partial \Omega$ belongs to the boundary of a side of the polytope $\Omega$, then the gradient degree of freedom $(\nabla I v)(\rv)$ is prescribed by $(I v)|_{\partial \Omega}$, and thus should depend only on $v|_{\partial \Omega}$. If $\rv$ belongs to the interior of a side of $\Omega$, then only the tangential part of $(\nabla I v)(\rv)$ is prescribed by $(I v)|_{\partial \Omega}$. Accordingly, we define $\delta_\rv := d$ in the first case and $\delta_\rv := d-1$ in the second case, and we assume that for $i=1,\delta_\rv$, the side $\sigma_{\rv,i}$ is included in $\partial \Omega$ and the vector $\bn_{\rv,i}$ is tangential to $\sigma_{\rv,i}$. This can always be achieved upon choosing the family of vectors $\bn_{\rv,i}$ appropriately. If $d=3$, then similarly for any boundary edge $e \in \cE_h$, $e \subset \partial \Omega$, we define $\delta_e := d-1$ if $e$ is included in the boundary of a side of $\Omega$, and $\delta_e := d-2$ if $e$ is included in the interior of a side of $\Omega$, and we assume that for $i=1,\delta_e$, the side $\sigma_{e,i}$ is included in $\partial \Omega$ and the vector $\bn_{e,i}$ is tangential to $\sigma_{e,i}$.

We refer to \cref{fig:interp} for an illustration of the above construction in dimension two.

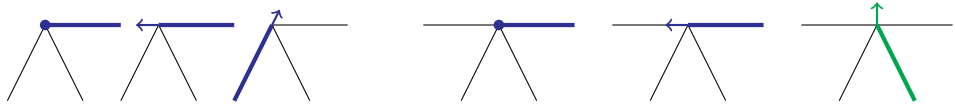
\begin{figure}[ht]
    \centering
    \begin{tikzpicture}
        \draw (0,0) -- (-0.5,-1) ;
        \draw (0,0) -- (0.5,-1) ;
        \color{Blue}
        \fill (0,0) circle (0.07) ;
        \draw[ultra thick] (0,0) -- (1,0) ;
        \color{black}
        \draw (1.5,0) -- (1,-1) ;
        \draw (1.5,0) -- (2,-1) ;
        \color{Blue}
        \draw[->, thick] (1.5,0) -- (1.2,0) ;
        \draw[ultra thick] (1.5,0) -- (2.5,0) ;
        \color{black}
        \draw (3,0) -- (4,0) ;
        \draw (3,0) -- (3.5,-1) ;
        \color{Blue}
        \draw[->, thick] (3,0) -- (3.1,0.2) ;
        \draw[ultra thick] (3,0) -- (2.5,-1) ;
        \color{black}
        \draw (5,0) -- (6,0) ;
        \draw (6,0) -- (5.5,-1) ;
        \draw (6,0) -- (6.5,-1) ;
        \color{Blue}
        \fill (6,0) circle (0.07) ;
        \draw[ultra thick] (6,0) -- (7,0) ;
        \color{black}
        \draw (7.5,0) -- (8.5,0) ;
        \draw (8.5,0) -- (8,-1) ;
        \draw (8.5,0) -- (9,-1) ;
        \color{Blue}
        \draw[->, thick] (8.5,0) -- (8.2,0) ;
        \draw[ultra thick] (8.5,0) -- (9.5,0) ;
        \color{black}
        \draw (10,0) -- (11,0) ;
        \draw (11,0) -- (12,0) ;
        \draw (11,0) -- (10.5,-1) ;
        \color{Green}
        \draw[->, thick] (11,0) -- (11,0.3) ;
        \draw[ultra thick] (11,0) -- (11.5,-1) ;
        \color{black}
    \end{tikzpicture}
    \caption{Illustration of a construction satisfying the condition \cref{eq:interp_boundary_dependence} in dimension two. Left: on vertices of $\Omega$, the value degree of freedom and both components of the gradient degrees of freedom need to be imposed by the boundary data. Right: on vertices of $\cT_h$ that belong to the interior of an edge of $\Omega$, only the value degree of freedom and the tangential part of the gradient degrees of freedom (in blue, as opposed to the normal part in green) need to be imposed by the boundary data.}
    \label{fig:interp}
\end{figure}

Let us now study the stability and the accuracy of the local interpolation operators $I_K^{\SZ}$ so defined. Our analysis is based on the following direct specialization of one of the inequalities in the norm equivalence result Lemma~4.7 in Ref.~\refcite{chen2022conforming}.

\begin{proposition}[discrete stability]
    \label{prop:norm_equivalence}
    Let $K \in \cT_h$. If $d = 2$, then for any $v \in V_{h,K}$,
    \begin{align*}
        \|v\|_{2,K}
        &\lesssim h_K^{-2} \|\Pi_{m-4}^{0,K} v\|_{0,K} + h_K^{-3/2} \sum_{e \in \cE_K} \|\Pi_{m-4}^{0,e} v\|_{0,e} + h_K^{-1/2} \sum_{e \in \cE_K} \|\Pi_{m-3}^{0,e} \partial_{\bn_{e,1}} v\|_{0,e} \\
        &\quad + h_K^{-1} \sum_{\rv \in \cV_K} |v(\rv)| + \sum_{v \in \cV_K} \sum_{i=1}^2 |\partial_{\bn_{\rv,i}} v(\rv)|.
    \end{align*}
    If $d = 3$, then for any $v \in V_{h,K}$,
    \begin{align*}
        \|v\|_{2,K}
        &\lesssim h_K^{-2} \|\Pi_{m-4}^{0,K} v\|_{0,K} + h_K^{-3/2} \sum_{F \in \cF_K} \|\Pi_{m-4}^{0,F} v\|_{0,F} + h_K^{-1/2} \sum_{F \in \cF_K} \|\Pi_{m-3}^{0,F} \partial_{\bn_F} v\|_{0,F} \\
        &\quad + h_K^{-1} \sum_{e \in \cE_K} \|\Pi_{m-4}^{0,e} v\|_{0,e} + \sum_{e \in \cE_K} \sum_{i=1}^2 \|\Pi_{m-3}^{0,e} \partial_{\bn_{e,i}} v\|_{0,e} \\
        &\quad + h_K^{-1/2} \sum_{\rv \in \cV_K} |v(\rv)| + h_K^{1/2} \sum_{v \in \cV_K} \sum_{i=1}^3 |\partial_{\bn_{\rv,i}} v(\rv)|.
    \end{align*}
\end{proposition}

For any $K \in \cT_h$, let $\omega_K := \intt (\bigcup \{\overline {K'} \mid K' \in \cT_h,\, \overline {K'} \cap \overline K \neq \emptyset\})$. We first prove the following stability result (that we call \emph{intermediate} since we will be able to improve it later, see \cref{coro:interp_stability_full_sz}).

\begin{lemma}[intermediate Scott-Zhang stability]
    \label{lemma:interp_stability_partial_sz}
    For any $v \in H^2(\Omega)$ and $K \in \cT_h$, one has
    \begin{equation*}
        \|I_K^{\SZ} v\|_{2,K} \lesssim  \sum_{i=0}^2 h_K^{i-2} |v|_{i, \omega_K}.
    \end{equation*}
\end{lemma}

\begin{proof}
    We apply the bounds in \cref{prop:norm_equivalence} to $I_K^{\SZ} v$ and estimate in turns all  terms in the right-hand side of such bounds. One has
    \begin{equation*}
        \|\Pi_{m-4}^{0,K} I_K^{\SZ} v\|_{0, K}
        = \|\Pi_{m-4}^{0, K} v\|_{0, K}
        \leq \|v\|_{0, K}
        \leq \|v\|_{0, \omega_K},
    \end{equation*}
    where the equality follows from the definition of $I_K^{\SZ} v$. Similarly, for any $\sigma \in \Sigma_K$, using the inverse trace inequality,
    \begin{align*}
        \|\Pi_{m-4}^{0,\sigma} I_K^{\SZ} v\|_{0, \sigma}
        &= \|\Pi_{m-4}^{0,\sigma} v\|_{0, \sigma}
        \leq \|v\|_{0, \sigma}
        \lesssim h_K^{-1/2} \|v\|_{0, \omega_K} + h_K^{1/2} |v|_{1, \omega_K}, \\
        \|\Pi_{m-3}^{0,\sigma} \partial_{\bn_\sigma} I_K^{\SZ} v\|_{0, \sigma}
        &= \|\Pi_{m-3}^{0,\sigma} \partial_{\bn_\sigma} v\|_{0, \sigma}
        \leq \|\partial_{\bn_\sigma} v\|_{0, \sigma}
        \lesssim h_K^{-1/2} |v|_{1, \omega_K} + h_K^{1/2} |v|_{2, \omega_K}.
    \end{align*}
    For any $\rv \in \cV_K$,
    \begin{align*}
        |(I_K^{\SZ} v)(\rv)|
        &= |(\Pi_m^{0, \sigma_\rv} v)(\rv)|
        \lesssim h_K^{-(d-1)/2} \|\Pi_m^{0, \sigma_\rv} v\|_{0, \sigma_\rv}
        \lesssim h_K^{-(d-1)/2} \|v\|_{0, \sigma_\rv} \\
        &\lesssim h_K^{-d/2} \|v\|_{0, \omega_K}, + h_K^{-(d-2)/2} |v|_{1, \omega_K},
    \end{align*}
    where we used successive trace and polynomial inverse inequalities for the first inequality. Similarly, for $i=1,d$,
    \begin{align*}
        |(\partial_{\bn_{\rv,i}} I_K^{\SZ} v)(\rv)|
        &= |(\Pi_{m-1}^{0, \sigma_{\rv,i}} \partial_{\bn_{\rv,i}} v)(\rv)|
        \lesssim h_K^{-(d-1)/2} \|\Pi_{m-1}^{0, \sigma_{\rv,i}} \partial_{\bn_{\rv,i}} v\|_{0, \sigma_{\rv,i}} \\
        &\lesssim h_K^{-d/2} |v|_{1, \omega_K}, + h_K^{-(d-2)/2} |v|_{2, \omega_K}.
    \end{align*}
    Finally, if $d=3$, we prove similarly that for $e \in \cE_K$ and $i=1,2$,
    \begin{align*}
        \|\Pi_{m-4}^{0,e} I_K^{\SZ} v\|_{0, e}
        &\lesssim h_K^{-1} \|v\|_{0, \omega_K} + |v|_{1, \omega_K}, \\
        \|\Pi_{m-3}^{0,e} \partial_{\bn_{e,i}} I_K^{\SZ} v\|_{0, e}
        &\lesssim h_K^{-1} |v|_{1, \omega_K} + |v|_{2, \omega_K}.
    \end{align*}
    This concludes the proof.
\end{proof}

By a Bramble-Hilbert argument, we deduce the following interpolation error estimate.

\begin{theorem}[Scott-Zhang interpolation estimate]
    \label{thm:interp_sz}
    Let $v \in H^2(\Omega)$. For any $K \in \cT_h$ and $1 \leq s \leq m$, if $v \in H^{s+1}(\omega_K)$, then
    \begin{equation*}
        \|v - I_K^{\SZ} v\|_{2,K} \lesssim h_K^{s-1} |v|_{s+1, \omega_K}.
    \end{equation*}
\end{theorem}

\begin{proof}
    By construction, one has $I_K^{\SZ} p = p$ for any $p \in \PP_m(\omega_K)$, and in particular for $p = \Pi_m^{0, \omega_K} v$. Then
    \begin{align*}
        \|v - I_K^{\SZ} v\|_{2, K}
        &\leq \|v - \Pi_m^{0, \omega_K} v\|_{2, K} + \|\Pi_m^{0, \omega_K} v - I_K^{\SZ} v\|_{2, K} \\
        &= \|v - \Pi_m^{0, \omega_K} v\|_{2, K} + \|I_K^{\SZ} (v - \Pi_m^{0, \omega_K} v)\|_{2, K} \\
        &\lesssim \sum_{i=0}^2 h_K^{i-2} \|v - \Pi_m^{0, \omega_K} v\|_{i, \omega_K},
    \end{align*}
    where in the last inequality we used \cref{lemma:interp_stability_partial_sz}, as well as the fact that $\|v - \Pi_m^{0, \omega_K} v\|_{2, K}$ is controlled by the term $\|v - \Pi_m^{0, \omega_K} v\|_{2, \omega_K}$. We conclude using classical Scott-Dupont theory\cite{brenner2008}.
\end{proof}

Let us mention that \cref{lemma:interp_stability_partial_sz} can be improved as follows once we have \cref{thm:interp_sz}:

\begin{corollary}[Scott-Zhang local stability]
    \label{coro:interp_stability_full_sz}
    For any $v \in H^2(\Omega)$ and $K \in \cT_h$, one has $\|I_K^{\SZ} v\|_{2,K} \lesssim \|v\|_{2, \omega_K}$.
\end{corollary}

\begin{proof}
    We use the triangle inequality $\|I_K^{\SZ} v\|_{2, K} \leq \|v\|_{2, K} + \|v - I_K^{\SZ} v\|_{2, K}$, and conclude by applying \cref{thm:interp_sz} with $s=1$.
\end{proof}

\subsubsection{Lagrange type interpolation}
\label{subsubsec:h2vem_lagrange}

Let us now provide some discussion of Lagrange type interpolation. In particular, in view of justifying our numerical results in \cref{sec:numerics}, we need to consider Lagrange interpolation in the minimal regularity case. This will lead to an error estimate involving both an $H^2$ and a $W^{1, \infty}$ term, see \cref{thm:interp_lagrange} below. This is unusual, but natural: $H^2$ is required because we are estimating an error in the $H^2$ norm, and $W^{1, \infty}$ (and even $C^1$) is required because Lagrange interpolation in our setting relies on pointwise evaluations of the gradient. Let us stress out that these considerations about minimal regularity are not specific to the virtual element method: in the simple setting of Lagrange interpolation for $C^0$ conforming finite elements, it would similarly be natural to estimate the interpolation error in a mix of the $H^1$ and $L^\infty$ norms.

For any $v \in H^2(\Omega) \cap C^1(\overline \Omega)$ and $K \in \cT_h$, we define $I_K^{\Lagrange} v \in V_{h, K}$ as the unique function in $V_{h, K}$ sharing with $v$ all the degrees of freedom defined in \cref{def:2Ddof} (or \cref{def:3Ddof} if $d=3$). We then define $I_h^{\Lagrange} v \in V_h$ by $(I_h^{\Lagrange} v)|_K = I_K^{\Lagrange} v$ in every $K \in \cT_h$.

Mimicking the analysis of Scott-Zhang interpolation, we start by proving the following lemma, and then we state our main interpolation estimate.

\begin{lemma}[intermediate Lagrange stability]
    \label{lemma:interp_stability_partial_lagrange}
    For any $v \in H^2(\Omega) \cap C^1(\overline \Omega)$ and $K \in \cT_h$, one has
    \begin{equation*}
        \|I_K^{\Lagrange} v\|_{2,K} \lesssim  \sum_{i=0}^1 h_K^{d/2+i-2} |v|_{i, \infty, K}.
    \end{equation*}
\end{lemma}

\begin{proof}
    We do not give the full details, since the proof follows the sketch of the one of \cref{lemma:interp_stability_partial_sz}. The main difference is that one cannot use $L^2$ based trace inequalities, which would require more differentiability of $v$ than we want to assume, so we rely instead on their $L^\infty$ based counterparts.
\end{proof}

\begin{theorem}[Lagrange interpolation with minimal regularity]
    \label{thm:interp_lagrange}
    Let $v \in H^2(\Omega) \cap C^1(\overline \Omega)$. For any $K \in \cT_h$ and $1 \leq s \leq m$, $0 \leq r \leq m$, if $v \in H^{s+1}(K) \cap W^{r+1,\infty}(K)$, then
    \begin{equation*}
        \|v - I_K^{\Lagrange} v\|_{2, K} \lesssim h_K^{s-1} |v|_{s+1, K} + h_K^{d/2+r-1} |v|_{r+1, \infty, K}.
    \end{equation*}
\end{theorem}

\begin{proof}
    By construction, one has $I_K^{\Lagrange} p = p$ for any $p \in \PP_m(K)$, and in particular for $p = \Pi_m^{0, K} v$. Then
    \begin{align*}
        \|v - I_K^{\Lagrange} v\|_{2, K}
        &\leq \|v - \Pi_m^{0, K} v\|_{2, K} + \|\Pi_m^{0, K} v - I_K^{\Lagrange} v\|_{2, K} \\
        &= \|v - \Pi_m^{0, K} v\|_{2, K} + \|I_K^{\Lagrange} (v - \Pi_m^{0, K} v)\|_{2, K} \\
        &\lesssim \|v - \Pi_m^{0, K} v\|_{2, K} + \sum_{i=0}^1 h_K^{d/2+i-2} |v - \Pi_m^{0, K} v|_{i, \infty, K},
    \end{align*}
    where we used \cref{lemma:interp_stability_partial_sz} in the last inequality. We conclude using classical Scott-Dupont theory\cite{brenner2008}.
\end{proof}

Similarly to the Scott-Zhang case, \cref{lemma:interp_stability_partial_lagrange} can now be improved as follows.

\begin{corollary}[Lagrange local stability]
    \label{coro:interp_stability_full_lagrange}
    For any $v \in H^2(\Omega) \cap C^1(\overline \Omega)$ and $K \in \cT_h$, one has $\|I_K^{\Lagrange} v\|_{2,K} \lesssim \|v\|_{2, K} + h_K^{d/2-1} |v|_{1, \infty, K}$.
\end{corollary}

\begin{proof}
    We use the triangle inequality $\|I_K^{\Lagrange} v\|_{2, K} \leq \|v\|_{2, K} + \|v - I_K^{\Lagrange} v\|_{2, K}$, and conclude by applying \cref{thm:interp_lagrange} with $s=1$ and $r=0$.
\end{proof}

Let us also state the following simplification of \cref{lemma:interp_stability_partial_lagrange} when $v \in W^{2, \infty}(K)$.

\begin{corollary}[Lagrange interpolation with $W^{2, \infty}$ regularity]
    \label{coro:interp_lagrange_simple}
    Let $v \in H^2(\Omega) \cap C^1(\overline \Omega)$. For any $K \in \cT_h$ and $1 \leq s \leq m$, if $v \in W^{s+1,\infty}(K)$, then
    \begin{equation*}
        \|v - I_K^{\Lagrange} v\|_{2, K} \lesssim h_K^{d/2+s-1} |v|_{s+1, \infty, K}.
    \end{equation*}
\end{corollary}

\section{Rates of convergence}
\label{sec:h2vem_rates}

In this section, we show how the error estimates derived in \cref{subsec:error_estimate} and \cref{sec:quadrature} can be combined with the interpolation results in \cref{subsec:h2vem_global_interp} in order to derive rates of convergence for the schemes \cref{eq:scheme,eq:quadr_scheme}. Throughout the section, we assume that we are working in the concrete VEM framework defined in \cref{sec:h2vem}.

As explained in \cref{rem:err_to_proj}, two different notions of error that make sense in the VEM setting are $\|u - u_h\|_2$ and $\|u - \Pi_m^* u_h\|_{2, h}$, so we state our results for the all-encompassing error $\|u - u_h\|_2 + \|u - \Pi_m^* u_h\|_{2, h}$.

Combining \cref{prop:err_to_proj} and \cref{thm:error_estimate,thm:interp_sz}, we can prove the following optimal error bound for the scheme \cref{eq:scheme}.

\begin{theorem}[rate of convergence]
    \label{thm:convergence_rate}
    Assume that $A$ satisfies the $\mu$-Cordes condition for some $0 \leq \mu < 1$, that $\gamma$ is a $\mu$-admissible scaling of $A$, that $\mu > c_*$ in \cref{eq:stab_scaling}, and that $g_h = I_h^{\SZ} g$. Let $u \in V$ and $u_h \in V_h$ denote respectively the unique solutions to \cref{eq:var} and \cref{eq:scheme}. If $u \in H^{s+1}(\Omega)$ for some $1 \leq s \leq m$, then
    \begin{equation*}
        \|u - u_h\|_2 + \|u - \Pi_m^* u_h\|_{2, h}
        \lesssim \frac{c_\pi c^*}{c_* - \mu} h^{s-1} |u|_{s+1}.
    \end{equation*}
\end{theorem}

\begin{proof}
    This follows from using \cref{prop:err_to_proj}, applying the error estimate \cref{eq:error_estimate} with $u_I = I_h^{\SZ} u$, and then using \cref{thm:interp_sz} and classical Scott-Dupont theory\cite{brenner2008} in order to estimate respectively the terms $\|u - I_h^{\SZ} u\|_2$ and $|u - \Pi_m^0 u|_{2, h}$.
\end{proof}

In the case of the scheme \cref{eq:quadr_scheme} involving numerical quadrature, one may similarly deduce the following error bound from \cref{prop:err_to_proj}, \cref{thm:quadr_error_estimate_discontinuous}, and \cref{coro:interp_lagrange_simple}.

\begin{theorem}[rate of convergence, with quadrature]
    \label{thm:quadr_convergence_rate}
    Assume that $A$ satisfies the $\mu$-Cordes condition everywhere for some $0 \leq \mu < 1$, that $\gamma$ is an everywhere $\mu$-admissible scaling of $A$, that $\mu > c_*$ in \cref{eq:stab_scaling}, that $g \in C^1(\overline \Omega)$, and that $g_h = I_h^{\Lagrange} g$. Let $u \in V$ and $u_h \in V_h$ denote respectively the unique solutions to \cref{eq:var} and \cref{eq:scheme}. If there exists an open set $E \subset \Omega$ satisfying \cref{eq:quadr_set_assumption} and if $u \in W^{s+1, \infty}(\cT_h)$ for some $1 \leq s \leq m$, then
    \begin{equation}
        \label{eq:quadr_convergence_rate}
        \|u - u_h\|_2 + \|u - \Pi_m^* u_h\|_{2, h}
        \lesssim \frac{c_\pi c^*}{c_* - \mu} h^{s-1} |u|_{s+1, \infty, h}.
    \end{equation}
\end{theorem}

\begin{proof}
    This follows from using \cref{prop:err_to_proj}, applying the error estimate \cref{eq:quadr_error_estimate_discontinuous} with $u_I = I_h^{\Lagrange} u$, and then using \cref{coro:interp_lagrange_simple} and classical Scott-Dupont theory in order to estimate respectively the terms $|u - I_h^{\Lagrange} u|_2$ and $|u - \Pi_m^0 u|_{2, h}$, $|u - \Pi_m^0 u|_{2, \infty, E \cap K}$.\iftoggle{author}{}{\\}
\end{proof}

Note that in the above theorem we had to assume $u \in W^{s+1, \infty}(\cT_h)$ rather than $u \in H^{s+1}(\cT_h)$, due to the rightmost term in the estimate \cref{eq:quadr_error_estimate_discontinuous}. One benefit of this assumption is that it implies that $u \in C^1(\overline \Omega)$ and thus allows to use Lagrange rather than Scott-Zhang interpolation, which is why the estimate \cref{eq:quadr_convergence_rate} only involves the regularity of $u$ in a broken norm. However, the assumption that $u \in W^{s+1, \infty}(\cT_h)$ with $s > 1$ may be too strong for some applications. It is possible to prove rates of convergence under weaker assumptions, but one may need to use ad hoc arguments. We give below an estimate that applies to the example in \cref{subsec:numerics_roughdata_roughsol}; see also \cref{subsec:numerics_singular} for an example in which $u \not\in W^{2, \infty}(\Omega)$.

In the estimate below, $u$ is assumed to have regularity $W^{s+1, \infty}$ on some submesh $\cT_h^* \subset \cT_h$, while also having $W^{2, \infty}$ regularity globally. This estimate is intended to be applied when the submesh $\cT_h^*$ is known to cover a large proportion of $\Omega$; for instance, as in \cref{subsec:numerics_roughdata_roughsol}, $\cT_h^*$ could be the set of all cells of $\cT_h$ except for those that intersect with a finite union of one-dimensional curves.

\begin{theorem}[rate of convergence, with quadrature and a rough solution]
    \label{thm:quadr_convergence_rate_alt}
    Assume that $A$ satisfies the $\mu$-Cordes condition everywhere for some $0 \leq \mu < 1$, that $\gamma$ is an everywhere $\mu$-admissible scaling of $A$, that $\mu > c_*$ in \cref{eq:stab_scaling}, that $g \in C^1(\overline \Omega)$, and that $g_h = I_h^{\Lagrange} g$. Let $u \in V$ and $u_h \in V_h$ denote respectively the unique solutions to \cref{eq:var} and \cref{eq:scheme}. Assume that there exists an open set $E \subset \Omega$ satisfying \cref{eq:quadr_set_assumption}, that $u \in W^{2, \infty}(\Omega)$, and that $u \in W^{s+1, \infty}(\cT_h^*)$ for some $1 \leq s \leq m$ and $\cT_h^* \subset \cT_h$. Then
    \begin{equation*}
        \|u - u_h\|_2 + \|u - \Pi_m^* u_h\|_{2, h}
        \lesssim \frac{c_\pi c^*}{c_* - \mu} \left(h^{s-1} |u|_{s+1, \infty, \cT_h^*} + \left(\sum_{K \in \cT_h \setminus \cT_h^*} |K|\right)^{1/2} |u|_{2, \infty}\right).
    \end{equation*}
\end{theorem}

\begin{proof}
    Let $E_h := \|u - u_h\|_2 + \|u - \Pi_m^* u_h\|_{2,h}$. We use \cref{prop:err_to_proj}, apply the error estimate \cref{eq:quadr_error_estimate_discontinuous} with $u_I = I_h^{\Lagrange} u$, and deduce that
    \begin{equation*}
        E_h \lesssim \frac{c_\pi c^*}{c_* - \mu} \left(\sum_{K \in \cT_h} (\|u - I_K^{\Lagrange} u\|_{2, K}^2 + |u - \Pi_m^0 u|_{2, K}^2 + |K|\, |u - \Pi_m^0 u|_{2, \infty, K}^2)\right)^{1/2}.
    \end{equation*}
    Using \cref{coro:interp_lagrange_simple} and classical Scott-Dupont theory in order to estimate respectively the terms $\|u - I_K^{\Lagrange} u\|_{2, K}^2$ and $|u - \Pi_m^0 u|_{2, K}^2$, $|u - \Pi_m^0 u|_{2, \infty, K}^2$,
    \begin{align*}
        E_h
        &\lesssim \frac{c_\pi c^*}{c_* - \mu} \left(\sum_{K \in \cT_h^*} h_K^{d+2s-2} |u|_{s+1, \infty, K}^2 + \sum_{K \in \cT_h \setminus \cT_h^*} h_K^d |u|_{2, \infty, K}^2\right)^{1/2} \\
        &\lesssim \frac{c_\pi c^*}{c_* - \mu} \left(\left(\sum_{K \in \cT_h^*} h_K^d\right)^{1/2} h^{s-1} |u|_{s+1, \infty, K} + \left(\sum_{K \in \cT_h \setminus \cT_h^*} h_K^d\right)^{1/2} |u|_{2, \infty, K}\right).
    \end{align*}
    We conclude using that $h_K^d \lesssim |K|$ and $\sum_{K \in \cT_h^*} h_K^d \lesssim 1$.
\end{proof}

%%%%%%%%%%%%%%%%%%%%%%%%%%%%%%%%%%%%%%%%%%%%%%%%%%%%%%%%%%%%%%%%%%%%%%%%%%%%%%%%%%%%%%
\section{Numerical experiments}
\label{sec:numerics}
%%%%%%%%%%%%%%%%%%%%%%%%%%%%%%%%%%%%%%%%%%%%%%%%%%%%%%%%%%%%%%%%%%%%%%%%%%%%%%%%%%%%%%

In this section, we apply the numerical scheme \cref{eq:quadr_scheme} to problems featuring various regularity properties. These problems originate from or are inspired by the ones in Ref.~\refcite{smears2013}. Our implementation is based on the code vem++\cite{dassi2023}. The sparse linear system resulting from the discretization is solved using a direct solver.

We use virtual element spaces featuring the polynomial consistency order $m = 2$. We use the barycenter quadrature rule: $Q_K[\phi] := |K| \phi(\overline x_K)$, where $\overline x_K$ denotes the barycenter of the cell $K \in \cT_h$. This is consistent with \cref{sec:quadrature}, in which, for the case $m = 2$, the quadrature rule $Q_K$ is only assumed to be exact for constants. We use a Lagrange interpolation of the boundary data: $g_I := I_h^{\Lagrange} g$.

For each considered problem, we compute the broken $H^2$-seminorm, $H^1$-seminorm, and $L^2$-norm errors between $u$ and $\Pi_2^* u_h$, where $u$ and $u_h$ are solution to respectively the continuous problem \cref{eq:var} and its discretization \cref{eq:quadr_scheme}~--- since the definitions of those errors involve integration of the exact solution $u$, in practice we approximate them using high-accuracy (fourth-order) quadrature rules.

We compare the $H^2$-seminorm errors with the predictions of the theory. Observe that we did not prove higher-order error estimates for the error in $H^1$ and $L^2$ norms; this would be nontrivial, since this would require elliptic regularity estimates for the dual equation to \cref{eq:main}. We still observe experimentally consistently higher orders of convergence in those norms.

\subsection{Example 1: rough coefficients and solution}
\label{subsec:numerics_roughdata_roughsol}

\begin{figure}[t!]
    \centering
    \includegraphics[width=\textwidth]{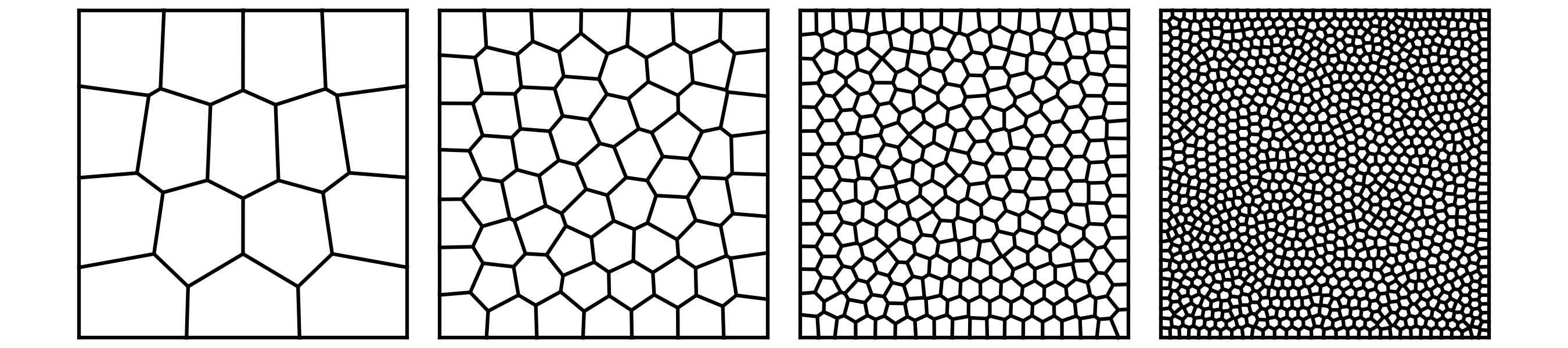}
    \caption{First four randomly generated polygonal meshes, in dimension two.}
    \label{fig:poly_meshes}
    \vspace{10pt}
    \includegraphics[width=\textwidth]{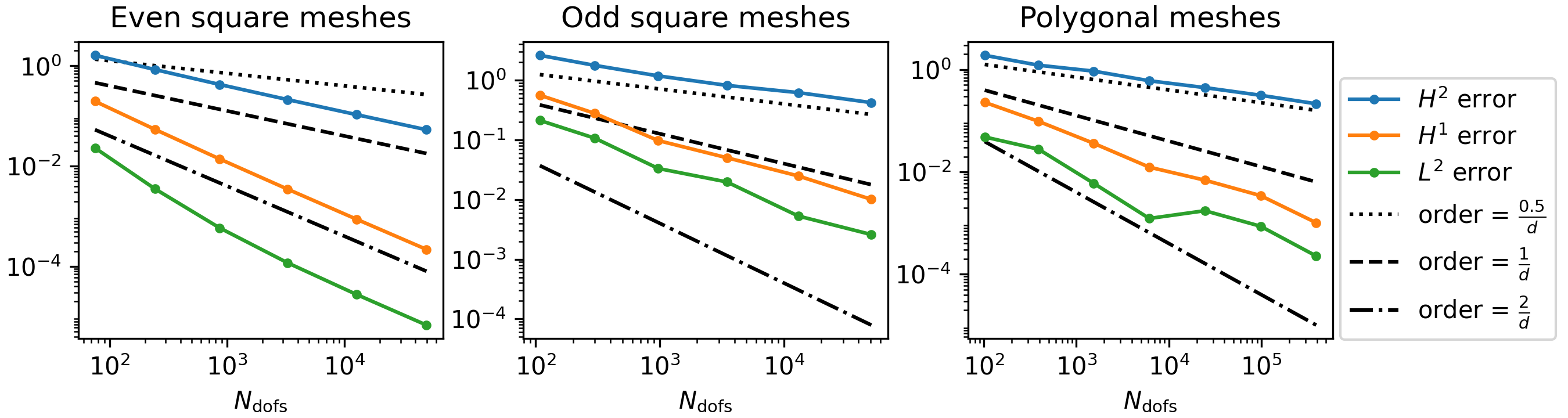}
    \includegraphics[width=\textwidth]{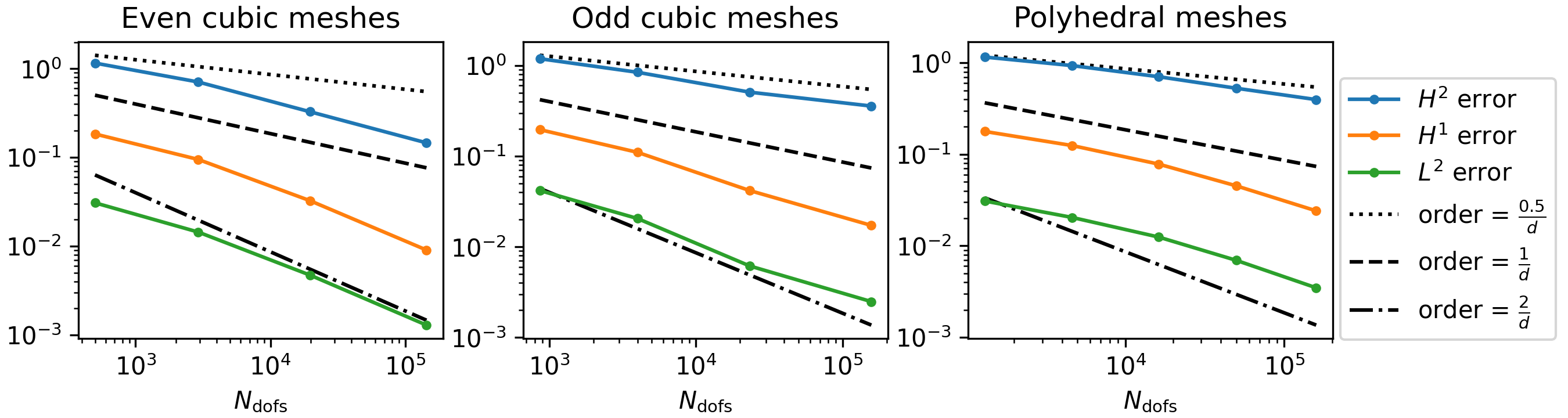}
    \caption{Example 1. Top: dimension two, bottom: dimension three. As expected, in the $H^2$ norm, the optimal rate of convergence is observed in the case of even square and cubic meshes, which are aligned with the discontinuities of $\nabla^2 u$, and convergence of order $1 / (2 d)$ with respect to the number of degrees of freedom is observed in the other cases.}
    \label{fig:errors}
\end{figure}

In dimensions $d \in \{2, 3\}$ and on the domain $\Omega := (-1, 1)^d$, we consider the problem \cref{eq:main} with
\begin{equation}
    \label{eq:pw_cst_coeffs}
    A(x) := \left((1 + (d-1) \delta_{ij}) \sgn_+(x_i) \sgn_+(x_j)\right)_{1 \leq i, j \leq d},
\end{equation}
where
\begin{equation*}
    \sgn_+(t) := \begin{cases}
        1 &\text{if } t \geq 0, \\
        -1 &\text{if } t < 0,
    \end{cases}
\end{equation*}
and with $g(x) := 0$ and $f$ chosen so that the solution $u$ is
\begin{equation*}
    u(x) = \prod_{i=1}^d \left(x_i e^{1 - |x_i|} - x_i\right).
\end{equation*}

The matrix field $A$ is piecewise constant and satisfies the $\mu$-Cordes condition with $\mu = \sqrt{2/5} \approx 0.63$ in dimension $d = 2$ and $\mu = \sqrt{6/11} \approx 0.74$ in dimension $d = 3$. According to \cref{prop:cordes_charac}, we let $\gamma(x) := \tr A(x) / |A(x)|^2 = d / (d^2 + d - 1)$. It is easily verified that $\gamma A \in [L^\infty(\Omega)]^{d \times d}$, $\gamma f \in L^\infty(\Omega)$, and $u \in W^{2, \infty}(\Omega)$.

We solve the scheme \cref{eq:quadr_scheme} on uniform square or cubic meshes obtained by dividing each edge of $\Omega$ into $N = 4, 8, 16, \ldots$ or $N = 5, 9, 17, \ldots$ intervals, and on randomly generated polytopal meshes, see \cref{fig:poly_meshes}. For even square and cubic meshes, one has $u \in C^\infty(\cT_h)$, thus \cref{thm:quadr_convergence_rate} applies and predicts convergence of order $1$ with respect to $h$ in the $H^2$ norm. For the other meshes, \cref{thm:quadr_convergence_rate_alt} applies with $\cT_h^* = \{K \in \cT_h \mid \forall x \in K,\, \forall 1 \leq i \leq d,\, x_i \neq 0\}$ and predicts convergence of order $1/2$ with respect to $h$ in the $H^2$ norm, since $\sum_{K \in \cT_h \setminus \cT_h^*} |K| \lesssim h$. The numerical results are consistent with those predictions, as illustrated in \cref{fig:errors}.

\subsection{Example 2: rough coefficients and smooth solution}
\label{subsec:numerics_roughdata_smoothsol}

As in the first problem, we consider equation \cref{eq:main} on the domain $\Omega = (-1, 1)^d$ in dimensions $d \in \{2, 3\}$, with coefficients given by \cref{eq:pw_cst_coeffs} and $g(x) := 0$, but now we choose $f$ such that the exact solution $u$ is
\begin{equation*}
    u(x) = \prod_{i=1}^d \sin(\pi x_i).
\end{equation*}

We solve the scheme \cref{eq:quadr_scheme} on the same sequences of meshes used for the first example. Since $u \in C^\infty(\Omega)$, \cref{thm:quadr_convergence_rate} now applies for all the meshes. Accordingly, we observe convergence of order $1$ with respect to $h$ in the $H^2$ norm for all the meshes, see \cref{fig:errorssmooth}.

\begin{figure}[t!]
    \centering
    \includegraphics[width=\textwidth]{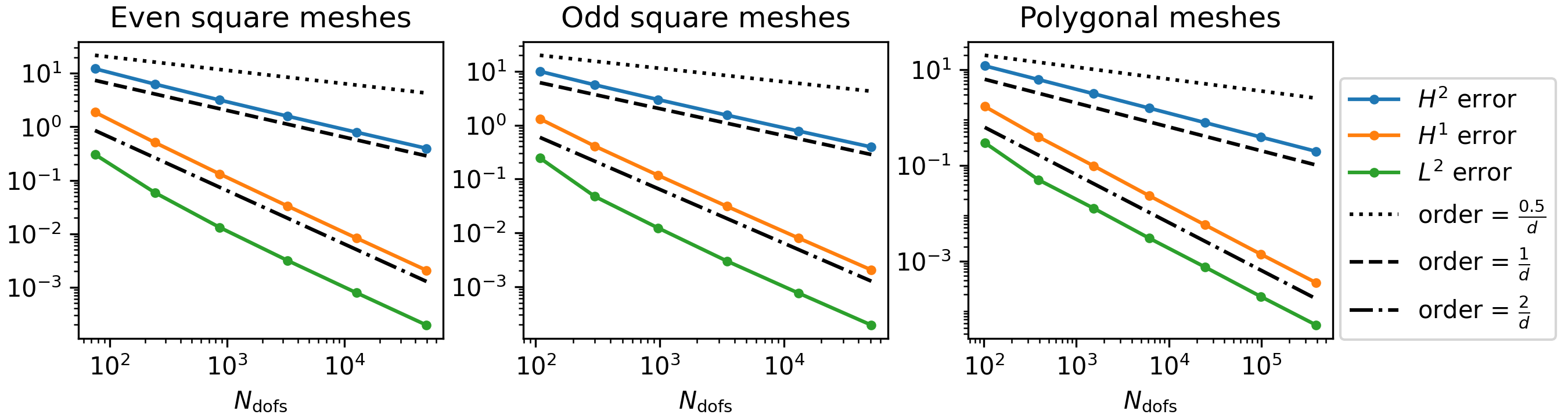}
    \includegraphics[width=\textwidth]{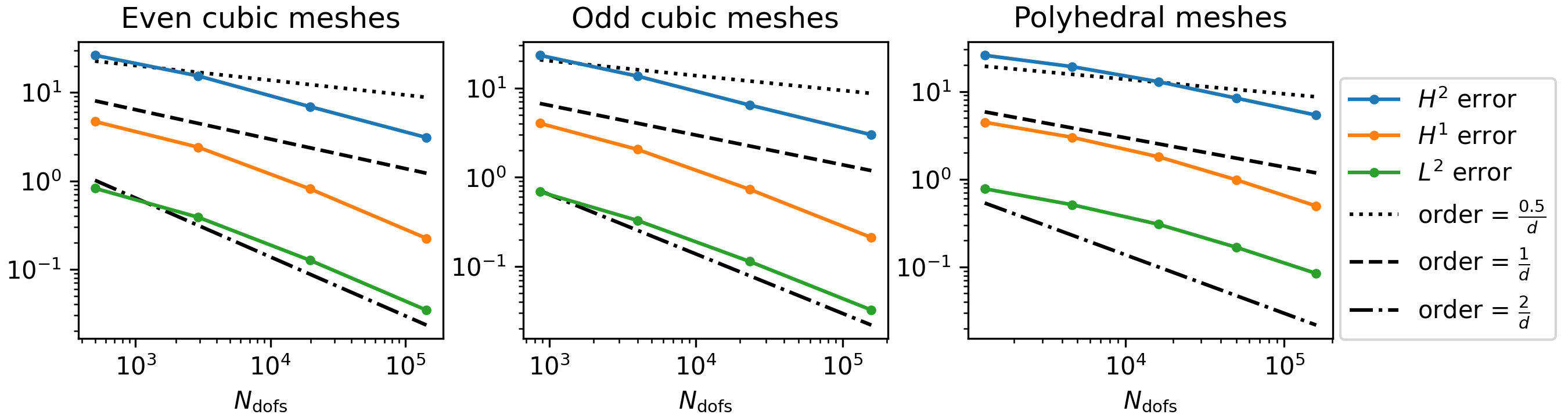}
    \caption{Example 2. Top: dimension two, bottom: dimension three. As expected, the optimal rate of convergence is observed in the $H^2$ norm regardless of the mesh. This highlights the fact that the discontinuity of $A$ and $f$ does not have any negative effect on the rate of convergence.}
    \label{fig:errorssmooth}
\end{figure}

\subsection{Example 3: point singularity and graded meshes}
\label{subsec:numerics_singular}

On the two-dimensional domain $\Omega := (0, 1)^2$, we consider the problem \cref{eq:main} with
\begin{align*}
    A(x) &:= \left(\delta_{ij} + \frac{x_i x_j}{|x|^2}\right)_{1 \leq i, j \leq 2}, &
    f(x) &:= 3.52 |x|^{-0.4},
\end{align*}
so that the solution is $u(x) = |x|^{1.6}$.

The matrix field $A$ satisfies the $\mu$-Cordes condition with $\mu = 1 / \sqrt{5}$. According to \cref{prop:cordes_charac}, we let $\gamma(x) := \tr A(x) / |A(x)|^2 = 3/5$. Observe that $u\in H^{2.6-\delta}(\Omega)$, for any $\delta > 0$.

We deduce from using \cref{prop:err_to_proj} and then applying \cref{eq:quadr_error_estimate_continuous} with $u_I = I_h^{\Lagrange} u$ that
\begin{align*}
    \|u - \Pi_2^* u_h\|_{2, h}
    &\lesssim \frac{c_\pi c^*}{c_* - \mu} \Bigg(\sum_{K \in \cT_h} (\|u - I_K^{\Lagrange} u\|_{2, K}^2 + |u - \Pi_2^0 u|_{2, h}^2
    \iftoggle{author}{}{\\ &\qquad \qquad \qquad}
    + Q_K[|\nabla^2 (u - \Pi_2^0 u)|^2])\Bigg)^{1/2}.
\end{align*}
In our experiments, we use meshes that have axis-aligned square cells, \emph{i.e.}~each cell $K \in \cT_h$ is of the form $K = (a_K, a_K + h_K / \sqrt{2}) \cap (b_K, b_K + h_K / \sqrt{2})$. We denote $x_K^0 := (a_K, b_K)$; remember also that $\overline x_K = (a_K + h_K / (2 \sqrt{2}), b_K + h_K / (2 \sqrt{2}))$ denotes the barycenter of $K$. Then,  we can prove the following.

\begin{figure}[t!]
    \centering
    \includegraphics[width=\textwidth]{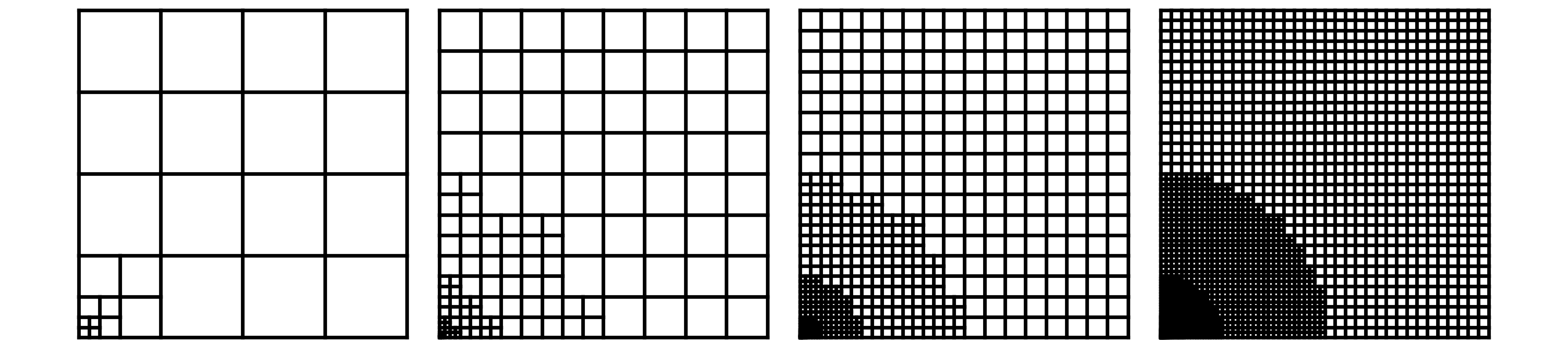}
    \caption{First four graded meshes.}
    \label{fig:graded_meshes}
    \vspace{10pt}
    \includegraphics[width=\textwidth]{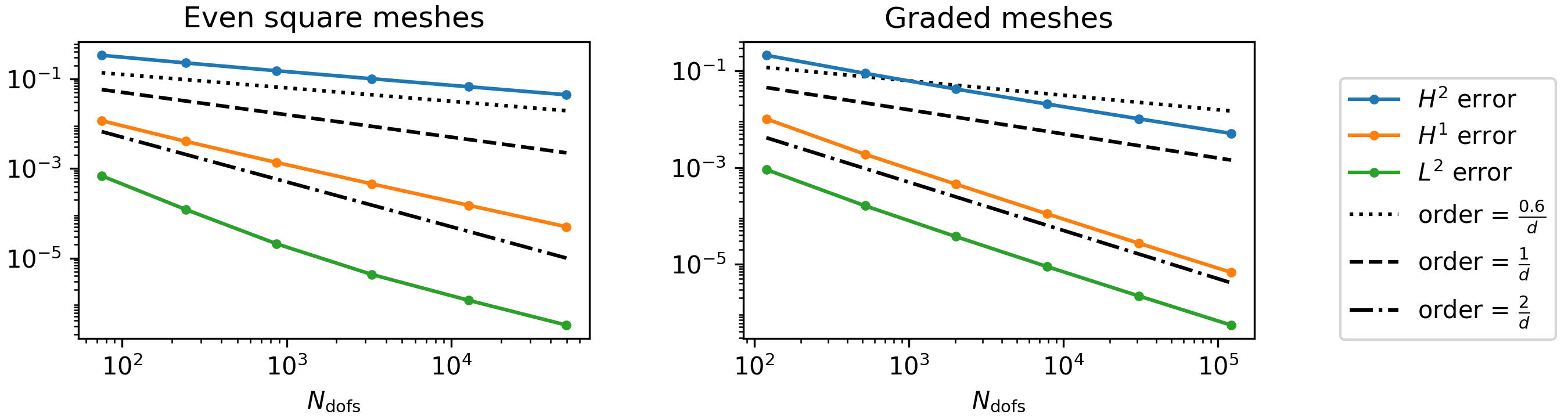}
    \caption{Example 3. In the $H^2$ norm and on uniform meshes, we observe convergence at the reduced rate $0.6 / d$ with respect to the number of degrees of freedom, due to the singularity of the problem. The optimal rate of convergence in the $H^2$ norm is recovered with meshes suitably graded towards the origin.}
    \label{fig:errorssingular}
\end{figure}

\begin{lemma}
    \label{lemma:numerics_singular}
    For  $u(x) = |x|^{1.6}$ and any $K \in \cT_h$, one has
    \begin{equation*}
        \|u - I_K^{\Lagrange} u\|_{2, K}^2 + |u - \Pi_2^0 u|_{2, K}^2 + Q_K[|\nabla^2 (u - \Pi_2^0 u)|^2] \lesssim h_K^4 |\overline x_K|^{-2.8}.
    \end{equation*}
\end{lemma}

\begin{proof}
    Let $E_K := \|u - I_K^{\Lagrange} u\|_{2, K}^2 + |u - \Pi_2^0 u|_{2, K}^2 + Q_K[|\nabla^2 (u - \Pi_2^0 u)|^2]$.

    If $|\overline x_K| \geq h_K$, then $|\overline x_K| \lesssim |x_K^0|$, and
    \begin{align*}
        E_K
        &\lesssim \|u - I_K^{\Lagrange} u\|_{2, K}^2 + h_K^2 |u - \Pi_2^0 u|_{2, \infty, K}^2
        \lesssim h_K^4 |u|_{3, \infty, K}^2
        = h_K^4 |\nabla^3 u(x_K^0)|^2 \\
        &\lesssim h_K^4 |x_K^0|^{-2.8}
        \lesssim h_K^4 |\overline x_K|^{-2.8},
    \end{align*}
    where we used \cref{coro:interp_lagrange_simple} and classical Scott-Dupont theory for the second inequality.

    Now assume that $|\overline x_K| \leq h_K$. Using that
    \begin{align*}
        Q_K[|\nabla^2 (u - \Pi_2^0 u)|^2]
        &\lesssim Q_K[|\nabla^2 u|^2] + Q_K[|\nabla^2 \Pi_2^0 u|^2]
        = Q_K[|\nabla^2 u|^2] + |\Pi_2^0 u|_{2, K}^2 \\
        &\lesssim Q_K[|\nabla^2 u|^2] + |u|_{2, K}^2
    \end{align*}
    together with \cref{thm:interp_lagrange} and classical Scott-Dupont theory, one has
    \begin{equation*}
        E_K \lesssim |u|_{1, \infty, K}^2 + |u|_{2, K}^2 + Q_K[|\nabla^2 u|^2].
    \end{equation*}
    We note that $Q_K[|\nabla^2 u|^2] \lesssim h_K^2 |\nabla^2 u(\overline x_K)|^2 \lesssim h_K^2 |\overline x_K|^{-0.8} \leq h_K^4 |\overline x_K|^{-2.8}$, and, by a direct computation, that $|u|_{1, \infty, K}^2 + |u|_{2, K}^2 \lesssim h_K^{1.2} \leq h_K^4 |\overline x_K|^{-2.8}$, which concludes the proof.
\end{proof}

The above result yields the error estimate
\begin{equation*}
    \|u - \Pi_2^* u_h\|_{2, h} \lesssim \frac{c_\pi c^*}{c_* - \mu} \left(\sum_{K \in \cT_h} h_K^4 |\overline x_K|^{-2.8}\right)^{1/2},
\end{equation*}
which we may further estimate depending on the mesh at hand.

We consider either meshes which are uniform or graded towards the origin.

On uniform meshes, using that $h \lesssim |\overline x_K|$ for all $K \in \cT_h$, one has
\begin{equation*}
    \|u - \Pi_2^* u_h\|_{2, h} \lesssim \frac{c_\pi c^*}{c_* - \mu} h^{0.6} \left(\sum_{K \in \cT_h} h^2 |\overline x_K|^{-2}\right)^{1/2},
\end{equation*}
where the sum inside the parentheses is similar to integrating $|x|^{-2}$ over $\Omega = (0, 1)^2$;
accordingly, we expect convergence of order $0.6$ with respect to $h$, up to a logarithm. The numerical results we obtain are consistent with this expectation, see \cref{fig:errorssingular}.

Inspired by \cref{lemma:numerics_singular} and following a principle of error equidistribution, we design graded meshes by recursive quadrisection towards the origin so that $h_K^4 |\overline x_K|^{-2.8}$ lies below some given threshold for all $K \in \cT_h$, see \cref{fig:graded_meshes}. According to the virtual element philosophy, we interpret hanging nodes as additional vertices of the polygonal cells. We observe from \cref{fig:errorssingular} that graded meshes allow recovering optimal convergence in the $H^2$ norm.

\section*{Acknowledgments}

The first author was partially supported by European Social Fund FSE SISSA 2019 Grant FP195673001 and NSF Grant DMS-1908267. The second author was partially supported by INdAM Research group GNCS. The third author was partially supported by NSF Grant DMS-1908267.

\iftoggle{author}{\bibliographystyle{plain}}{\bibliographystyle{ws-m3as}}
\bibliography{references}
\end{document}